\documentclass[11pt,reqno]{amsart}

\usepackage{amsmath,amssymb,amsthm,graphicx,calligra,mathrsfs,tikz-cd,mathtools,cleveref,url,filecontents}
\usepackage[none]{hyphenat}

\def\mf{\mathfrak}
\def\ds{\displaystyle}
\newcommand{\id}[1]{\langle {#1}\rangle}
\newtheorem{theorem}{Theorem}[section]

\newtheorem{example}[theorem]{Example}
\newtheorem{proposition}[theorem]{Proposition}
\newtheorem{lemma}[theorem]{Lemma}
\newtheorem{corollary}[theorem]{Corollary}
\theoremstyle{definition}
\newtheorem{defn}[theorem]{Definition}
\newcommand{\stor}{\mathscr{T}{\kern -3pt {or}}}
\newcommand{\sext}{\mathscr{E}{\kern -1pt {xt}}}
\newcommand{\Ext}{\mathscr{E}{\kern -1pt {xt}}}
\newcommand{\shom}{\mathscr{H}{\kern -3pt {om}}}
\newcommand{\bj}{\mathbf{j}}
\newcommand{\Gm}{G^{\text{min}}}

\newcommand{\bb}[1]{\mathbb{#1}}
\newcommand{\idshf}[2]{\mathcal{O}_{#1}(-#2)}
\newcommand{\xyp}[1]{(\bar{#1}_P)_\bb{P}}

\newcommand{\hx}{\hat{\xi}^w_P}
\crefformat{equation}{(#2#1#3)}
\crefrangeformat{equation}{(#3#1#4)--(#5#2#6)}

\renewcommand{\ss}[1]{\mathcal{O}_{#1}}
\renewcommand{\bb}[1]{\mathbb{#1}}

\AtBeginDocument{%
	\setlength\abovedisplayskip{10pt}
	\setlength\belowdisplayskip{10pt}
	\setlength\abovedisplayshortskip{2pt}
	\setlength\belowdisplayshortskip{10pt}}
\begin{document}

\title{Positivity in $T$-Equivariant $K$-theory of partial flag varieties associated to
Kac-Moody groups}

\author{Joseph Compton and Shrawan Kumar}

\maketitle

% header =================================================================

%\noindent\emph{Joseph Compton} \hfill \emph{compja@live.unc.edu} \\
%\emph {The University of North Carolina at Chapel Hill} 
%\hfill
%\emph{https://tarheels.live/josephcompton} 

%\svs
%\hrule, Preprint 
%\svs

\noindent
{\bf Abstract:} We prove sign-alternation of the product structure constants in the basis  dual to the basis consisting of the structure sheaves of Schubert varieties in the torus-equivariant Grothendieck group of coherent sheaves on the partial flag varieties $G/P$ associated to an arbitrary symmetrizable Kac-Moody group $G$, where $P$ is any parabolic subgroup of finite type.  This extends the previous work of Kumar from $G/B$ to $G/P$. When $G$ is of finite type, i.e., it is a semisimple group, then it was proved by Anderson-Griffeth-Miller.

\section{Introduction}

This is a continuation of Kumar \cite{Ku}, the notation of which we freely use. 

Let $G$ be any symmetrizable Kac-Moody group over $\mathbb{C}$ completed along the negative roots and $G^{\text{min}}\subset G$ the `minimal' Kac-Moody group as in \cite[\S 7.4]{K}. Let $B$ be the standard (positive) Borel subgroup, $B^-$ the standard negative Borel subgroup, $H=B\cap B^-$ the standard maximal torus and $W$ the Weyl group. Let $\bar{X} = G/B$ be the `thick' flag variety (introduced by Kashiwara) which contains the standard flag variety $X=G^{\text{min}}/B$. Let $T$ be the adjoint torus, i.e.,  $T:=H/Z(\Gm)$, where $Z(\Gm)$ denotes the center of $\Gm$ and let $R(T)$ denote the representation ring of $T$. For any $w\in W$ (the Weyl group), we have the Schubert cell $C_w:=BwB/B\subset X$, the Schubert variety $X_w := \overline{C_w} \subset X$, the opposite Schubert cell $C^w:=B^-wB/B\subset \bar{X}$, and the opposite Schubert variety $X^w: = \overline{C^w}\subset\bar{X}$. When $G$ is a (finite-dimensional) semisimple group, it is referred to as the finite case. \\
\indent Let $P$ be any standard parabolic subgroup of $G$ of finite type (i.e., the Levi component of $P$ is finite dimensional) and set $W^P$ to be the set of minimal length coset representatives of $W/W_P$, where $W_P$ is the Weyl group of $P$. Then, for any $w\in W^P$, we have the Schubert cell $C_w^P: = BwP/P\subset{X_P}$, Schubert variety $X_w^P: = \overline{C_w^P}\subset{X_P}$, the opposite Schubert cell $C^w_P:=B^-wP/P\subset \bar{X}_P$, and the opposite Schubert variety $X^w_P:= \overline{C^w_P}\subset \bar{X}_P$ where $X_P=\Gm/P$ and $\bar{X}_P = G/P$. \\
\indent Let $K^0_T(\bar{X}_P)$ be the Grothendieck group of $T$-equivariant  coherent $\ss{\bar{X}_P}$-modules of the (in general) nonquasi-compact scheme $\bar{X}_P$ (cf. Definition 3.1). Let $\{[\xi^w_P]\}_{w\in W^P}$ be the `basis' of $K^0_T(\bar{X}_P)$ given by $\xi^w_P:=\idshf{X^w_P}{\partial X^w_P}$ (where  $\partial X^w_P:=  X^w_P\setminus C^w_P$) 
and express the product in $K^0_T(\bar{X}_P)$ in this `basis': \begin{equation}[\xi^u_P]\cdot[\xi^v_P] = \sum_{w\in W^P}d_{u,v}^w(P)[\xi^w_P],\,\,\, \text{ for some unique} \,d_{u,v}^w(P)\in R(T).\label{product}
\end{equation}
The above sum, in general, is infinite. 

The following result is our main theorem (cf. Theorem \ref{mainthm}). This was first conjectured by Graham-Kumar \cite{GK} in the finite case, proven in this case by Anderson-Griffeth-Miller \cite{AGM}, and then proven in the general Kac-Moody case for $G/B$ by Kumar \cite{Ku}. 
\begin{theorem} For any $u,v,w\in W^P$,
	\begin{equation*}(-1)^{l(u)+l(v)+l(w)}d_{u,v}^w(P)\in \mathbb{Z}_{\geq 0}[(e^{-\alpha_1}-1),\cdots, (e^{-\alpha_r}-1)],
	\end{equation*} where $\{\alpha_1,\cdots, \alpha_r\}$ are the simple roots, i.e., $(-1)^{l(u)+l(v)+l(w)}d_{u,v}^w(P)$ is a polynomial in the variables $x_1 = e^{-\alpha_1}-1,\cdots , x_r = e^{-\alpha_r}-1$ with nonnegative integral coefficients. 
	\end{theorem}
	
The proof of the above theorem generally follows the proof in \cite{Ku} for the full flag variety $\bar{X}$. However, several of the geometric and cohomological results in the case of the full flag variety had to be generalized to the partial flag varieties $\bar{X}_P$. In particular, 
the definition of the basis in terms of the dualizing sheaves had to be modified in the parabolic case as the dualizing sheaf in this general case is not suitable.

In another related work, Baldwin-Kumar \cite{BaKu} proved an analogue of the above theorem for the basis consisting of the structure sheaves of the opposite Schubert varieties.

We have added an appendix  to determine the dualizing sheaf $\omega_{X^w_P}$ of the Cohen-Macaulay scheme $X^w_P \subset \bar{X}_P$  for any $w\in W^P$. Even though this result is not used in the paper, we believe that it is interesting on its own.

\vskip4ex
\noindent
{\bf Acknowledgements:} The second author was partially supported by the NSF grant  DMS-1802328. We thank the referee for useful comments.

\section{Notation}\label{notation}

We work over the field $\mathbb{C}$ of complex numbers. By a variety, we mean an algebraic variety over $\mathbb{C}$ which is reduced, but not necessarily irreducible. For a scheme $X$ and a closed subscheme $Y$, $\idshf{X}{Y}$ denotes the ideal sheaf of $Y$ in $X$. \\
\indent Let $G$ be any symmetrizable Kac-Moody group completed along the negative roots (as opposed to along the positive roots as in \cite[Chapter 6]{K}), and let $\Gm\subset G$ be the `minimal' Kac-Moody group as in \cite[\S 7.4]{K}. Let $B$ be the standard (positive) Borel subgroup, $B^-$ the standard negative Borel subgroup, $H=B\cap B^-$ the standard maximal torus and $W$ the Weyl group \cite[Chapter 6]{K}. For any standard parabolic subgroup $P$ of $G$ of finite type (i.e., the Levi component of $P$ is finite dimensional),
 let \begin{equation*}\bar{X}_P = G/P
\end{equation*} be the `thick' partial flag variety which contains the standard Kac-Moody  partial flag ind-variety \begin{equation*}X_P = \Gm/P.
\end{equation*}

If $G$ is not of finite type, then $\bar{X}_P$ is an infinite-dimensional nonquasi-compact scheme (cf. \cite[$\S$4]{Ka} for the case of $P=B$; the case of general $P$ is similar) and $X_P$ is an ind-projective variety \cite[\S 7.1]{K}. The group $\Gm$ (in particular, the maximal torus $H$) acts on $X_P$ and $\bar{X}_P$. Let $T$ be the quotient $H/Z(\Gm)$, where $Z(\Gm)$ is the center of $\Gm$. Then the action of $H$ on $\bar{X}_P$ (and $X_P$) descends to an action of $T$. 

Let $W^P$ denote the set of minimal length coset representatives in the quotient $W/W_P$,  where $W_P\subset W$ is the Weyl group of $P$. Then, for any $w\in W^P$,  we have the Schubert cell 
\begin{equation*}C_w^P : = BwP/P\subset X_P,
\end{equation*}
the Schubert variety \begin{equation*}X_w^P: = \overline{C_w^P} = \bigsqcup_{\substack{u\leq w\\u\in W^P}} C_u^P\subset X_P,
\end{equation*} the opposite Schubert cell \begin{equation*}C^w_P : = B^-wP/P\subset \bar{X}_P,
\end{equation*} and the opposite Schubert variety \begin{equation*}X^w_P:=\overline{C^w_P} = \bigsqcup_{\substack{u\geq w\\u\in W^P}} C^u_P \subset \bar{X}_P,
\end{equation*} all endowed with the reduced subscheme structures. When $P=B$, we drop the qualification $P$, thus $\bar{X}=\bar{X}_B, X_w= X_w^B$ etc. Then, $X_w^P$ is a (finite-dimensional) irreducible projective subvariety of $X_P$ and $X^w_P$ is a finite-codimensional irreducible subscheme of $\bar{X}_P$. We recall the well-known result that for $v<w \in W^P$, there exists a chain $v=v_0< v_1 < \dots <v_n=w$ in $W^P$ such that $l(v_{i+1})=l(v_i)+1$ for all $0\leq i\leq n-1$ (see, e.g., Theorem 2.5.5 in {\it Combinatorics of Coxeter Groups} by A. Bj\"orner and F. Brenti as pointed out by R. Proctor).
We denote by $Z_w$ the Bott-Samelson-Demazure-Hansen (BSDH) variety as in \cite[\S 7.1]{K}, which is a $B$-equivariant desingularization of $X_w^P$ \cite[Proposition 7.1.15]{K}. Further, $X_w^P$ is normal and has rational singularities 
(in particular, Cohen-Macaulay)
\cite[Theorem 8.2.2]{K}. \\
We also define the boundary of the Schubert variety by \begin{equation*}\partial X_w^P := X_w^P\ \backslash\  C_w^P
\end{equation*} with the reduced subscheme structure. For any $u,w\in W^P$ with $u\leq w$, we have the Richardson variety \begin{equation*}X_w^u(P):=X^u_P\cap X_w^P\subset X_P
\end{equation*} 
endowed with the reduced subvariety structure. By \cite[Proof of Proposition 5.3]{KuS} together with \cite[Lemma 1.1.8]{BK}, $X_P$ is Frobenius split in any characteristic $p>0$ compatibly splitting $\{X_w^P, X^w_P\cap X_P\}_{w\in W^P}$. In particular, any scheme-theoretic intersection 
$ X^P_{w_1} \cap \dots \cap X^P_{w_m}\cap X_P^{v_1}\cap \dots \cap X_P^{v_n}$ (for $m \geq 1$)  is  {\it reduced}. 

We denote by $Z_w^u$ the $T$-equivariant desingularization of $X_w^u$ as in \cite[Theorem 6.8]{Ku}. 
For $u,w \in W^P$, the canonical projection map $X^u_w\to X^u_w(P)$ is a proper birational morphism, where $C^u\cap C_w$ maps isomorphically onto a $T$-stable open subset of  $C^u_P\cap C^P_w$ (cf. Lemma \ref{rich:props} and {\cite[Lemma 4.3]{Deo}).
 Hence, $Z^u_w$ is a $T$-equivariant desingularization of $X^u_w(P)$. 

We denote the representation ring of $T$ by $R(T)$. Let $\{\alpha_1,\cdots, \alpha_r\}\subset\mf{h}^*$ be the set of simple roots, $\{\alpha_1^\vee,\cdots, \alpha_r^\vee\}\subset \mf{h}$ the set of simple coroots, and $\{s_1,\cdots, s_r\}\subset W$ the corresponding set of simple reflections, where $\mf{h} := \text{Lie}(H)$. Let $\rho\in\mf{h}^*$ be any integral weight satisfying \begin{equation*}\rho(\alpha_i^\vee)=1,\ \text{for all }\ 1\leq i\leq r.
\end{equation*}

 If $G$ is a finite-dimensional semisimple group, $\rho$ is unique, but for a general KM group $G$, it may not be unique. Denote by $\rho_Y:=\sum_{i\in Y} \varpi_i$, where $\{\varpi_1,\cdots,\varpi_r\}$ are fixed fundamental weights and $Y\subset\{1,\cdots, r\}$ corresponds to the Levi component of $P$, i.e., $\{\alpha_i\}_{i\in Y}$ are the simple roots of the Levi component $L$ of $P$ containing $H$. (Observe that, in general,  $\varpi_i$ are not unique.) 
  We then define the integral weight $\hat{\rho}_Y\in\mf{h}^*$ by \begin{equation*}\hat{\rho}_Y:= \rho-\rho_Y.
\end{equation*} 

For any integral weight $\lambda$, let $\mathbb{C}_\lambda$ denote the one-dimensional representation of $H$ on $\mathbb{C}$ given by $h\cdot z = \lambda(h)z$ for $h\in H$ and $z\in\bb{C}$. This uniquely extends to a representation of $B$. We call $\lambda$ a {\it $P$-weight} if 
$\dot{\lambda}(\alpha_i^\vee)=0$ for all $i\in Y$, where $\dot{\lambda}$ is the derivative of $\lambda$.  If $\lambda$ is a $P$-weight, this action extends uniquely to a representation  of $P$ and we define the $G$-equivariant line bundle $\mathcal{L}^P(\lambda)$ on $\bar{X}_P$ by \begin{equation*} \mathcal{L}^P(\lambda):= G\times^P\bb{C}_{-\lambda},
\end{equation*} where for any representation $V$ of $P$, $G\times^P V := (G\times V)/P$ where $P$ acts on $G\times V$ by $(g,v)\cdot p = (gp,p^{-1}v)$ for $g\in G,\ v\in V$ and $p\in P$. We also define the line bundle \begin{equation*}e^\lambda:=\bar{X}_P\times \bb{C}_\lambda,
\end{equation*} which while trivial as a non-equivariant line bundle, is equivariantly non-trivial with the diagonal action of $H$.

\section{Identification of the dual of the structure sheaf basis}

\begin{defn} For a quasi-compact scheme $Y$, an $\mathcal{O}_Y$-module $\mathcal{S}$ is called \emph{coherent} if it is finitely presented as an $\mathcal{O}_Y$-module and any $\ss{Y}$-submodule of finite type admits a finite presentation. 
	
	A subset $S\subset W^P$ is called an \emph{ideal} if $x\in S$ and $v\leq x$ imply $v\in S$. An $\ss{\bar{X}_P}$-module $\mathcal{S}$ is called \emph{coherent} if $\mathcal{S}|_{V^S}$ is a coherent $\ss{V^S}$-module for any finite ideal $S\subset W^P$, where $V^S$ is the quasi-compact open subset of $\bar{X}_P$ defined by \begin{equation*}V^S := \bigcup_{w\in S} B^-wP/P = \bigcup_{w\in S} wU^-P/P,
	\end{equation*} where $U^-$ is the unipotent part of $B^-$. Observe that $\{ wU^-P/P\}_{w\in W^P}$ is an open cover of $\bar{X}_P$. 	Let $K^0_T(\bar{X}_P)$ denote the Grothendieck group of $T$-equivariant coherent $\ss{\bar{X}_P}$-modules $\mathcal{S}$. Since the coherence condition on $\mathcal{S}$ is imposed only for $\mathcal{S}|_{V^S}$ for finite ideals $S$, $K^0_T(\bar{X}_P)$ can be thought of as the inverse limit of $K^0_T(V^S)$ as $S$ varies over all finite ideals of $W^P$ \cite[\S 2]{KS}.
	
	Similarly, define $K_0^T(X_P):=\text{Lim}_{n\to\infty} K_0^T(X^P_n)$,  where $\{X^P_n\}_{n\geq 1}$ is the filtration of $X_P$ giving $X_P$ its ind-projective variety structure (i.e.,  $X^P_n = \bigcup_{l(w)\leq n,w\in W^P} BwP/P$) and $K_0^T(X^P_n)$ is the Grothendieck group of $T$-equivariant coherent sheaves on the projective variety $X_n^P$. 
	\end{defn}

\noindent For any $w\in W^P$,  we denote the class of  the $T$-equivariant coherent sheaf $\ss{X_w^P}$ by
\begin{equation*} [\ss{X_w^P}]\in K_0^T(X_P).
\end{equation*}
\begin{lemma} $\{[\ss{X_w^P}]\}_{w\in W^P}$ forms a basis of $K_0^T(X_P)$ as an $R(T)$-module. 
	\end{lemma}
\begin{proof} Apply \cite[\S 5.2.14 and Theorem 5.4.17]{CG}.
	\end{proof}

For any $u\in W^P$, $\ss{X_P^u}$ is a coherent $\ss{\bar{X}_P}$-module. In particular, $\ss{\bar{X}_P}$ is a coherent $\ss{\bar{X}_P}$-module. We record the following lemma due to Kashiwara-Shimozono \cite[Proof of Lemma 8.1]{KS}. 

\begin{lemma}\label{KS:res} Any $T$-equivariant coherent sheaf $\mathcal{S}$ on $V^u$ admits a finite free resolution in $\textup{Coh}_T(\ss{V^u})$ (for some $n>0$ depending upon $\mathcal{S}$): \begin{equation*}0\to S_n\otimes \ss{V^u}\to \cdots \to S_1\otimes \ss{V^u}\to S_0\otimes\ss{V^u}\to \mathcal{S}\to 0,
	\end{equation*} where $S_k$ are finite-dimensional $T$-modules, $V^u:=uU^-P/P\subset \bar{X}_P$ is quasi-compact, and $\textup{Coh}_T(\ss{V^u})$ is the abelian category of $T$-equivariant coherent $\ss{V^u}$-modules. 
		\end{lemma}
	
	We define a pairing 
	\begin{equation}\label{pair}\id{\ ,\ }:K^0_T(\bar{X}_P)\otimes K_0^T(X_P)\to R(T),\hspace{1cm} \id{[\mathcal{S}],[\mathcal{F}]} = \sum_i (-1)^i\chi_T(X^P_n,\stor_i^{\ss{\bar{X}_P}}(\mathcal{S},\mathcal{F})),
	\end{equation}
 where $\mathcal{S}$ is a $T$-equivariant coherent sheaf on $\bar{X}_P$ and $\mathcal{F}$ is a $T$-equivariant coherent sheaf on $X_P$ supported on $X_n^P$ for some $n$, and where $\chi_T$ is the $T$-equivariant Euler-Poincar\'e characteristic. This pairing is well-defined by \cite[Proof of Lemma 3.5]{Ku}. 
 
 \begin{defn} For any $u\in W^P$, define the $T$-equivariant sheaf on $\bar{X}_P$,
 	\begin{equation*}\xi^u_P:=\idshf{X^u_P}{\partial X^u_P}, \,\,\text{where $\partial X^u_P:= X^u_P\setminus C^u_P$ with the reduced sub-scheme structure.}
 	\end{equation*}
 	\end{defn}
 
 \begin{proposition} $\{[\xi^u_P]\}_{u\in W^P}$ forms an infinite  basis of $K^0_T(\bar{X}_P)$ as an $R(T)$-module, where by {\rm infinite basis} we mean that $K^0_T(\bar{X}_P) = \prod_{u\in W^P}\, R(T) [\xi^u_P] $. 
 	\end{proposition}
\begin{proof} By the same proof as in \cite[$\S$2]{KS}  $\{[\ss{X_P^v}]\}_{v\in W^P}$ is an infinite  basis of $K^0_T(\bar{X}_P)$. Moreover,  $\ds [\xi^u_P]=[\ss{X^u_P}]+\sum_{\substack{w'>u\\ w'\in W^P}} r_{w'}[\ss{X^{w'}_P}]$ for some $r_{w'}\in R(T)$, the proposition follows.
	\end{proof}

\begin{proposition}\label{prop:finunion} For any finite union $Y=\bigcup_{i=1}^k X^{v_i}_P$ of opposite Schubert varieties and any  $w$ in $W^P$,
	\begin{enumerate}\item[(a)] $\stor_j^{\ss{\bar{X}_P}}(\ss{Y},\ss{X_w^P}) =0$ for all $j>0$.
		\item[(b)] $H^j(X_n^P,\ss{Y\cap X_w^P})=0$ for all $j>0$, where $n$ is any positive integer such that $X_n^P\supset X_w^P$. 
	\end{enumerate}
	In particular, this applies to the case $Y=\partial X^u_P$.
\end{proposition}
\begin{proof} This follows by the same argument as \cite[Proof of Corollary 5.7]{Ku}. Observe that $X^u_w(P)$ is irreducible by Lemma 
\ref{rich:props}.	Moreover, the vanishing $H^j(X_n^P, \ss{X_w^v(P)})=0$, for all $j>0$, follows from the corresponding result for $\bar{X}$ 
(cf. \cite[Corollary 3.2]{KuS})
and applying the Leray spectral sequence for the bundle $\pi:\bar{X} \to \bar{X}_P$, since $\pi^{-1}(X^v_w(P)) = X^v_{ww_o^P}$, where 
$w_o^P$ is the longest element of $W_P$. 
\end{proof}

\begin{proposition}\label{prop:torvan} For any $v, w\in W^P$, 
	\[\stor_j^{\ss{\bar{X}_P}}(\xi^v_P,\ss{X_w^P}) = 0\hspace{5mm} \text{for all}\ j>0.\]
	\end{proposition} 
\begin{proof} By Proposition \ref{prop:finunion}(a),
 we have the vanishing \[\stor_j^{\ss{\bar{X}_P}}(\ss{X^v_P},\ss{X_w^P}) = 0\hspace{5mm} \text{for all}\ j>0.\] The proposition follows from this together with \Cref{prop:finunion} part (a) applied to $Y=\partial X^v_P$ and the long exact sequence for $\stor$ associated to the short exact sequence \[0\to \xi^v_P\to\ss{X^v_P}\to \ss{\partial X^v_P}\to 0.\qedhere\]
	\end{proof}

\begin{proposition}\label{duality} For any $u, v\in W^P$, \begin{equation*}\id{[\xi^u_P],[\ss{X_v^P}]}=\delta_{u,v}.
	\end{equation*}
	\end{proposition}
\begin{proof} For any $u, w\in W^P$, the pairing is by definition 
	\[\id{[\xi^u_P],[\ss{X_w^P}]}=\sum_{i}(-1)^i\chi_T(X^P_n,\stor_i^{\ss{\bar{X}_P}}(\xi^u_P,\ss{X_w^P})),\]
	with $n$ taken such that $n\geq l(w)$. By \Cref{prop:torvan}, this becomes \begin{equation}\label{eulchar1}\id{[\xi^u_P],[\ss{X_w^P}]}=\chi_T(X_n^P, \xi^u_P\otimes_{\ss{\bar{X}_P}}\ss{X_w^P}).
		\end{equation} From \Cref{prop:finunion}(a) and the definition  $\xi^u_P := \idshf{X^u_P}{\partial X^u_P}$, we have the sheaf exact sequence 
	\[0\to \xi^u_P\otimes_{\ss{\bar{X}_P}}\ss{X_w^P}\to \ss{X^u_P}\otimes_{\ss{\bar{X}_P}}\ss{X_w^P}\to \ss{\partial X^u_P}\otimes_{\ss{\bar{X}_P}}\ss{X_w^P}\to 0.\] Hence,
	\begin{equation}\label{eulchar2}\chi_T(X_n^P, \xi^u_P\otimes_{\ss{\bar{X}_P}}\ss{X_w^P}) = \chi_T(X_n^P, \ss{X^u_w(P)})-\chi_T(X_n^P, \ss{(\partial X^u_P) \cap X_w^P}). 
		\end{equation}
	From \Cref{rich:props}, the Richardson variety $X^u_w(P)$ is irreducible (when nonempty, i.e., $u\leq w$) and hence $(\partial X^u_P)\cap X_w^P= \bigcup_{u<v\leq w} X^v_w(P)$ is connected (when nonempty) since $wP\in X^v_w(P)$ for all $u<v\leq w$. If $u\not\leq w$, $X^u_w(P)$ is empty, and hence  \Cref{eulchar1,eulchar2} imply $\id{[\xi^u_P],[\ss{X_w^P}]}=0$. Hence, we assume $u\leq w$ so that $X^u_w(P)$ is nonempty. By Proposition \ref{prop:finunion} (b),	
	\[H^i(X_n^P, \ss{X^u_w(P)}) = 0\hspace{5mm}\text{for all}\ i>0.\] 
(It should be mentioned that \Cref{rich:props}	does not depend upon the proof of this proposition or any other result in this paper dependent upon this proposition, thus avoiding any circularity of argument.)

	Further, by \Cref{prop:finunion} 
	(b) applied to $\partial X^u_P$, \[H^i(X_n^P,\ss{(\partial{X^u_P})\cap X_w^P}) =0\hspace{5mm}\text{for all}\ i>0.\] Thus, for $u\leq w$, \[\chi_T(X_n^P, \ss{X^u_w(P)})=1\] and for $u<w$, \[\chi_T(X_n^P, \ss{(\partial X^u_P)\cap X_w^P}) = 1.\] Therefore, when $u<w$, $\id{[\xi^u_P],[\ss{X_w^P}]} =0$. Finally, if $u=w$, we have $\id{[\xi^w_P],[\ss{X_w^P}]} =1$ since $\partial X^u_P\cap X^P_w$ is empty in this case. 		
	\end{proof}
Let $\Delta:X_P\to X_P\times X_P$ be the diagonal map. Express the coproduct in $K_0^T(X_P)$: \begin{equation*} \Delta_*[\ss{X_w^P}]=\sum_{u,v\in W^P}q_{u,v}^w(P)[\ss{X_u^P}]\otimes [\ss{X_v^P}],\,\,\,\text{for $q^w_{u,v}(P)\in R(T)$}.
\end{equation*} Also, express the product in $K_T^0(\bar{X}_P)$
\begin{equation}\label{coeff:d}[\xi^u_P]\cdot [\xi^v_P]=\sum_{w\in W^P} d_{u,v}^w(P) [\xi^w_P],\,\,\,\text{for $d^w_{u,v}(P) \in R(T)$}.\end{equation}
\begin{proposition}\label{identify:qd} For any $u,v,w\in W^P$, \begin{equation*}q_{u,v}^w(P)=d_{u,v}^w(P).
	\end{equation*}
	\end{proposition}
\begin{proof} Let $\bar{\Delta}:\bar{X}_P\to \bar{X}_P\times\bar{X}_P$ be the diagonal map. Then, for any $w\in W^P$, 
	\begin{align*} \id{\bar{\Delta}^*[\xi^u_P\boxtimes \xi^v_P],[\ss{X_w^P}]} & = \id{[\xi^u_P\boxtimes \xi^v_P],\Delta_*[\ss{X_w^P}]}\\ & = \id{[\xi^u_P\boxtimes \xi^v_P],\sum_{u',v'\in W^P} q_{u',v'}^w(P)[\ss{X_{u'}^P}]\otimes [\ss{X_{v'}^P}]}\\& = q_{u,v}^w(P), \hspace{1cm} \text{by Proposition \ref{duality}}.
		\end{align*}
	On the other hand, since $[\xi^u_P]\cdot[\xi^v_P] = \bar{\Delta}^*[\xi^u_P\boxtimes\xi^v_P],$ we also have 
	\begin{align*}\id{\bar{\Delta}^*[\xi^u_P\boxtimes \xi^v_P],[\ss{X_w^P}]} & = \id{[\xi^u_P]\cdot[\xi^v_P],[\ss{X_w^P}]}\\ & = \left\langle\sum_{w'\in W^P} d_{u,v}^{w'}(P) [\xi^{w'}_P],[\ss{X_w^P}]\right\rangle\\ & = d_{u,v}^w(P), \hspace{1cm} \text{by Proposition \ref{duality}}\qedhere
		\end{align*}
	\end{proof}

\section{The Mixing Space and Mixing Group}\label{mixing}

In this section, we introduce the mixing space $(X_P)_\bb{P}$, which is a bundle over a product of projective spaces with fiber $X_P$. This allows the reduction from the $T$-equivariant $K$-theory to the non-equivariant $K$-theory. We then introduce the mixing group $\Gamma$ whose action is sufficient to allow for a transversality result used to prove part of our main technical result. 

Fix any positive integer $N$ and let $\bb{P} :=(\bb{P}^N)^r$, where $r=\dim T$. {\it Unless otherwise explicitly stated, $N$ is any positive integer.} Observe that $\mathbb{P}$ depends upon the choice of $N$. For any $\bj= (j_1,\cdots, j_r) \in [N]^r$, where $[N] :=\{0,1,\cdots, N\}$, set \begin{equation*}\bb{P}^\bj = \bb{P}^{N-{j_1}}\times\cdots\times \bb{P}^{N-{j_r}} \text{ and } \bb{P}_\bj = \bb{P}^{j_1}\times\cdots\times \bb{P}^{j_r}. 
\end{equation*} We also define the boundary of $\bb{P}_\bj$ by 
\begin{equation*}\partial\bb{P}_\bj:=(\bb{P}^{j_1-1}\times\bb{P}^{j_2}\times\cdots\times\bb{P}^{j_r})\cup\cdots\cup (\bb{P}^{j_1}\times\bb{P}^{j_2}\times\cdots\times\bb{P}^{j_r-1}),
\end{equation*} where we interpret $\bb{P}^{-1}:=\varnothing$ as the empty set. Throughout the paper, we fix an identification $T\cong (\bb{C}^*)^r$ under $t\mapsto (e^{\alpha_1}(t),\cdots,e^{\alpha_r}(t))$.

\begin{defn} Let $E(T)_\bb{P}:=(\bb{C}^{N+1}\setminus\{0\})^r$ be the total space of the standard principal $T$-bundle $E(T)_\bb{P}\to \bb{P}$. Let $\pi_{X_P}:(X_P)_\bb{P}:=E(T)_\bb{P}\times^T X_P\to \bb{P}$ be the fibration with fiber $X_P$ associated to the principal $T$-bundle $E(T)_\bb{P}\to \bb{P}$, where we twist the action of $T$ on $X_P$ via \begin{equation}t\odot x = t^{-1}x.\label{twist}
	\end{equation} For any $T$-subscheme $Y\subset X_P$, denote $Y_\bb{P}:=E(T)_\bb{P}\times^T Y\subset (X_P)_\bb{P}$.
	\end{defn}
The following proposition follows easily by using \cite[$\S$5.2.14 and Theorem 5.4.17]{CG} applied to the vector bundle $(BwP/P)_{\bb{P}}\to \bb{P}$. 
\begin{proposition} $K_0((X_P)_\bb{P}):=\textup{Lim}_{n\to\infty} K_0((X^P_n)_\bb{P})$ is a free module over the ring $K_0(\bb{P})=K^0(\bb{P})$ with basis $\{[\ss{(X_w^P)_\bb{P}}]\}_{w\in W^P}$. Therefore, $K_0((X_P)_\bb{P})$ has a $\bb{Z}$-basis \begin{equation*}\{\pi_{X_P}^*([\ss{\bb{P}^\bj}])\cdot[\ss{(X_w^P)_\bb{P}}]\}_{\bj\in[N]^r,\ w\in W^P},
	\end{equation*} where we view $[\ss{\bb{P}^\bj}]$ as an element of $K_0(\bb{P})=K^0(\bb{P})$. 
	\end{proposition}

Let $Y_P:=X_P\times X_P$. The diagonal map $\Delta:X_P\to Y_P$ gives rise to the embedding \begin{equation*}\tilde{\Delta}:(X_P)_\bb{P}\to (Y_P)_\bb{P}=E(T)_\bb{P}\times^TY_P\cong (X_P)_\bb{P}\times_\bb{P}(X_P)_\bb{P}.
\end{equation*} Therefore, we have (denoting the projection $(Y_P)_\bb{P}\to\bb{P}$ by $\pi_{Y_P}$) \begin{equation}\label{equivcoeff:c}\tilde{\Delta}_*[\ss{(X_w^P)_\bb{P}}]=\sum_{\substack{u,v\in W^P\\ \bj\in[N]^r}} c_{u,v}^w(\bj)\pi^*_{Y_P}([\ss{\bb{P}^\bj}])\cdot [\ss{(X_u^P\times X_v^P)_\bb{P}}]\in K_0((Y_P)_\bb{P}) 
\end{equation} for some $c_{u,v}^w(\bj)\in\bb{Z}$. 

Let $\bar{Y}_P=\bar{X}_P\times\bar{X}_P$ and let $K^0((\bar{Y}_P)_\bb{P})$ denote the Grothendieck group associated to the semi-group of coherent $\ss{(\bar{Y}_P)_\bb{P}}$-modules. Define, for $u,v\in W^P$, 
 \[\widetilde{\xi^u_P\boxtimes \xi^v_P}:= \idshf{(X^u_P\times X^v_P)_\bb{P}}{\partial((X^u_P\times X^v_P)_\bb{P})} \in K^0((\bar{Y}_P)_\bb{P}),\]
where $\partial((X^u_P\times X^v_P)_\bb{P}):=((\partial X^u_P\times X^v_P)\cup (X^u_P\times\partial X^v_P))_\bb{P}$.

\begin{lemma} \label{lemma:c} With the notation as above, \begin{equation}\label{coeff:c}c_{u,v}^w(\bj) = \id{\pi_{\bar{Y}_P}^*[\idshf{\bb{P}_\bj}{\partial \bb{P}_\bj}]\cdot[\widetilde{\xi^u_P\boxtimes\xi^v_P}],\tilde{\Delta}_*[\ss{(X_w^P)_\bb{P}}]},
		\end{equation} where $\pi_{\bar{Y}_P}:(\bar{Y}_P)_\bb{P}\to \bb{P}$ is the projection and the pairing $\id{\ ,\ }:K^0((\bar{Y}_P)_\bb{P})\otimes K_0((Y_P)_\bb{P})\to \bb{Z}$ is defined similar to \eqref{pair} above. Explicitly, \begin{equation} \id{[\mathcal{S}],[\mathcal{F}]}=\sum_i (-1)^i\chi((\bar{Y}_P)_\bb{P},\stor_i^{\ss{(\bar{Y}_P)_\bb{P}}}(\mathcal{S},\mathcal{F})),
	\end{equation}
where $\chi$ denotes the (non-equivariant) Euler-Poincar\'e characteristic.
	\end{lemma}
\begin{proof} We compute 
	\begin{align*} &\id{\pi_{\bar{Y}_P}^*[\idshf{\bb{P}_\bj}{\partial \bb{P}_\bj}]\cdot[\widetilde{\xi^u_P\boxtimes\xi^v_P}],\tilde{\Delta}_*[\ss{(X_w^P)_\bb{P}}]} \\
		&\hspace{3mm} = \id{\pi_{\bar{Y}_P}^*[\idshf{\bb{P}_\bj}{\partial \bb{P}_\bj}]\cdot[\widetilde{\xi^u_P\boxtimes\xi^v_P}],\sum_{\substack{u',v'\in W^P\\ \bj'\in[N]^r}} c_{u',v'}^w(\bj')\pi^*_{Y_P}([\ss{\bb{P}^{\bj'}}])\cdot [\ss{(X_{u'}^P\times X_{v'}^P)_\bb{P}}]}\\
		&\hspace{3mm} = c_{u,v}^w(\bj)\hspace{5mm}\text{by \Cref{duality} and \cite[identity (20)]{Ku}}.\qedhere
		\end{align*}
	\end{proof}

\begin{defn} Let $T$ act on $B$ via \[t\cdot b= t^{-1}bt\] where $t\in T$ and $b\in B$. Then there is a natural action of $\Delta T$ on $B\times B$. Let $(B^2)_\bb{P}$ be the ind-group scheme over $\bb{P}$: \[(B^2)_\bb{P}:=E(T)_\bb{P}\times^T(B\times B)\to \bb{P}\] and let $\Gamma_0$ denote the group of global sections of $(B^2)_\bb{P}$ under pointwise multiplication. Thus, $\Gamma_0$ can be identified with the set of regular maps $f: E(T)_\bb{P} \to B\times B$ such that $f(e\cdot t) =t^{-1}\cdot f(e)$ for all $e\in E(T)_\bb{P}$ and $t\in T$. Now, $GL(N+1)^r$ acts canonically on $(B^2)_\bb{P}$ in a way compatible with its action on $\bb{P}$ and acts on $\Gamma_0$ via its pull-back.
Thus, $GL(N+1)^r$ normalizes $\Gamma_0$, where the conjugation action is explicitly given by (for $g\in GL(N+1)^r$ and $\gamma_0 \in \Gamma_0$):
$$(g\gamma_0g^{-1})[e, b]= [e, \bar{\gamma}_0(g^{-1}e)\cdot b]$$
for $e\in E(T)_{\bb{P}}$ and $b\in B\times B$, where $\gamma_0[e, b]= [e, \bar{\gamma}_0(b)]$. 
We define the mixing group $\Gamma=\Gamma_{B^2}:=\Gamma_0\rtimes GL(N+1)^r$: \[1\to \Gamma_0\to\Gamma\to GL(N+1)^r\to 1.\] By the comments following \cite[Lemmas 4.7 and 4.8]{Ku}, we have the following two lemmas. 
	\end{defn}
\begin{lemma} \label{connected} $\Gamma$ is connected.
	\end{lemma}
\begin{lemma}\label{fiber} Given any $\bar{e}\in \bb{P}$ and any $(b_1,b_2)$ in the fiber of $(B^2)_\bb{P}$ over $\bar{e}$, there exists a section $\gamma\in\Gamma_0$ such that $\gamma(\bar{e})=(b_1,b_2)$. 
	\end{lemma}

Define the action of $\Gamma$ on $(Y_P)_\bb{P}$ by \[(\gamma,g)\cdot [e,(y,y')]=[ge,\gamma(ge)\cdot(y,y')]\] for $\gamma\in \Gamma_0$, $g\in GL(N+1)^r$, $e\in E(T)_\bb{P}$, and $(y,y')\in {Y}_P$, where the action of $\Gamma_0$ is via the standard action of $B^2$ on ${Y}_P$. From \Cref{fiber}, it follows that the orbits of the $\Gamma$-action on $(Y_P)_\bb{P}$ are precisely $\{(C_u^P\times C_v^P)_\bb{P}\}_{u,v\in W^P}$.

\begin{proposition}\label{product} For any coherent sheaf $\mathcal{S}$ on $\bb{P}$, and any $u,v\in W^P$, \[\pi^*[\mathcal{S}]\cdot[\widetilde{\xi^u_P\boxtimes\xi^v_P}]=[\pi^*(\mathcal{S})\otimes_{\ss{\xyp{Y}}}(\widetilde{\xi^u_P\boxtimes\xi^v_P})]\in K^0(\xyp{Y}),\] where we abbreviate $\pi_{\bar{Y}_P}$ by $\pi$ and $\pi^*(\mathcal{S})=\ss{\xyp{Y}}\otimes_{\ss{\bb{P}}}\mathcal{S}$. In particular, \[\pi^*[\idshf{\bb{P}_\bj}{\partial\bb{P}_\bj}]\cdot[\widetilde{\xi^u_P\boxtimes\xi^v_P}]=[\pi^*(\idshf{\bb{P}_\bj}{\partial\bb{P}_\bj})\otimes_{\ss{\xyp{Y}}}(\widetilde{\xi^u_P\boxtimes\xi^v_P})].\]
	\end{proposition} 

\begin{proof} This follows from the same argument as that of \cite[Propositon 4.9]{Ku}.
	\end{proof}

\section{Statement of Main Results}

The following is our main technical result. The proof of its two parts (A) and (B) are given in Sections 6 and 10 respectively. 

\begin{theorem}\label{main} Let $N$ be any positive integer. For general $\gamma\in \Gamma$, $\bj\in [N]^r$, and any $u,v,w\in W^P$,
	\begin{enumerate}\item[(A)] $\stor_i^{\ss{\xyp{Y}}}(\pi^*(\idshf{\bb{P}_\bj}{\partial \bb{P}_\bj})\otimes(\widetilde{\xi^u_P\boxtimes\xi^v_P}),\gamma_*\tilde{\Delta}_*(\ss{(X_w^P)_\bb{P}}))=0$ for all $i>0$, where $\gamma\in\Gamma$ is viewed as an automorphism of the scheme $(Y_P)_{\mathbb{P}}$. \\
		
		\item[(B)] For $c_{u,v}^w(\bj)\neq 0$, where $c_{u,v}^w(\bj)$ is defined by \eqref{equivcoeff:c}		
		(see also \eqref{coeff:c}), \[H^p(\xyp{Y},\pi^*(\idshf{\bb{P}_\bj}{\partial \bb{P}_\bj})\otimes(\widetilde{\xi^u_P\boxtimes\xi^v_P})\otimes\gamma_*\tilde{\Delta}_*(\ss{(X_w^P)_\bb{P}}))=0\] for all $p\neq |\bj|+l(w)-(l(u)+l(v))$, where $|\bj|=\sum_{i=1}^r j_i$. 
		\end{enumerate}
	\end{theorem}

Since $\Gamma$ is connected (cf. \Cref{connected}), we have the following result as an immediate corollary of \Cref{lemma:c}, \Cref{product}, and \Cref{main}.

\begin{corollary}\label{coeff:calts} $(-1)^{|\bj|+l(w)-(l(u)+l(v))}c_{u,v}^w(\bj)\in \bb{Z}_{\geq 0}$.
	\end{corollary}

Recall the definition of the structure constants $d_{u,v}^w(P)\in R(T)$ (cf. \eqref{coeff:d}) for the product in $K^0_T(\bar{X}_P)$. 

\begin{lemma} \label{lemma5.4}
For any $u,v,w\in W^P$, $d^w_{u, v}(P)\in \mathbb{Z}[ (e^{-\alpha_1}-1), \dots, (e^{-\alpha_r}-1)]$.
\end{lemma}
\begin{proof} By \cite[Proposition 3.5]{GK}, which is valid in the symmetrizable Kac-Moody case by the same proof, for any $u,v,w\in W^P$, 
\begin{equation} \label{eqn5.3} d^w_{u,v}(P)= \sum_{u'\in uW_P , \\ v'\in vW_P}\, d^w_{u', v'}(B).
\end{equation}
Now, $d^w_{u', v'}(B)\in  \mathbb{Z}[ (e^{-\alpha_1}-1), \dots, (e^{-\alpha_r}-1)]$ by \cite[Lemma 4.12]{Ku}. Hence, the lemma follows from the identity \eqref{eqn5.3}.
\end{proof}

The following lemma follows easily from the identity \eqref{equivcoeff:c}, \Cref{identify:qd}, \cite[Lemma 6.2]{GK} (which is valid in the Kac-Moody case as well) (see also \cite[\S 3]{AGM}).

\begin{lemma}\label{lemma:offset} For any $u,v,w\in W^P$, we can choose a large enough $N_1$ (depending on $u,v,w$) and write for any $N\geq N_1$ 
(cf. Lemma \ref{lemma5.4})
\[d_{u,v}^w(P) = \sum_{\bj} d_{u,v}^w(\bj) (e^{-\alpha_1}-1)^{j_1}\cdots (e^{-\alpha_r}-1)^{j_r}\] for some unique $d_{u,v}^w(\bj)\in\bb{Z}$, where $\bj=(j_1,\cdots, j_r) \in [N]^r$. Then \begin{equation}\label{coeff:offset} d_{u,v}^w(\bj) = (-1)^{|\bj|}c_{u,v}^w(\bj).
		\end{equation}
	\end{lemma}

The following main theorem of this paper is an immediate consequence of \Cref{coeff:calts} and \Cref{lemma:offset}, which was  proved in the $G/B$ case by Kumar \cite{Ku}. 

\begin{theorem} \label{mainthm} For any symmetrizable Kac-Moody group $G$ and parabolic subgroup $P$ of finite type,  and any $u,v,w\in W^P$, 
the structure constants in $K^0_T(\bar{X}_P)$ satisfy \begin{equation}(-1)^{l(u)+l(v)+l(w)} d_{u,v}^w(P) \in \bb{Z}_{\geq 0}[(e^{-\alpha_1}-1),\cdots,(e^{-\alpha_r}-1)].
		\end{equation}
	\end{theorem}

\section{Proof of part (A) of \Cref{main}}

\begin{proof}[Proof of Theorem 5.1(A)] Note that since the assertion is local in $\bb{P}$, we assume $\xyp{Y}\simeq \bb{P}\times\bar{Y}_P$. Then,
	\begin{align}\pi^*\idshf{\bb{P}_\bj}{\partial\bb{P}_\bj}\simeq  \idshf{\bb{P}_\bj}{\partial\bb{P}_\bj}\boxtimes \ss{\bar{Y}_P},\label{identity:idshf}\\
	\widetilde{\xi^u_P\boxtimes\xi^v_P}\cong \ss{\bb{P}}\boxtimes (\xi^u_P\boxtimes\xi^v_P),\label{identity:tilde}\\
		\ss{(X_w^P\times X_w^P)_\bb{P}}\simeq \ss{\bb{P}}\boxtimes (\ss{X_w^P}\boxtimes\ss{X_w^P}).\label{identity:schub}
		\end{align}
	We first show that for any $\ss{(Y_w^P)_\bb{P}}$-module $\mathcal{S}$ (where $(Y_w^P)_\bb{P}:=(X_w^P\times X_w^P)_\bb{P}$), 
	\begin{multline}\label{identity:stors}\stor_i^{\ss{\xyp{Y}}}(\pi^*(\idshf{\bb{P}_\bj}{\partial \bb{P}_\bj})\otimes(\widetilde{\xi^u_P\boxtimes\xi^v_P}),\mathcal{S})\\\simeq \stor_i^{\ss{(Y_w^P)_\bb{P}}}(\ss{(Y_w^P)_\bb{P}}\otimes_{\ss{\xyp{Y}}} \left(\pi^*(\idshf{\bb{P}_\bj}{\partial \bb{P}_\bj})\otimes(\widetilde{\xi^u_P\boxtimes\xi^v_P})\right),\mathcal{S}).
		\end{multline}
	To prove \eqref{identity:stors}, we observe the following: Let $R,S$ be commutative rings with a ring homomorphism $R\to S$, $M$ an $R$-module and $N$ an $S$-module. Then $N\otimes_S(S\otimes_R M) \simeq N\otimes_R M$. This gives the following isomorphism, provided $\textup{Tor}_j^R(S,M) = 0$ for all $j>0$:
	 \begin{equation}\label{isom:Tor} \textup{Tor}^R_i(M,N)\simeq \textup{Tor}^S_i(S\otimes_R M,N).
		\end{equation}
	Then \eqref{identity:stors} follows with $R = \ss{\xyp{Y}},\ S = \ss{(Y_w^P)_\bb{P}},\ M=\pi^*(\idshf{\bb{P}_\bj}{\partial \bb{P}_\bj})\otimes(\widetilde{\xi^u_P\boxtimes\xi^v_P}),$ and $N=\mathcal{S}$, provided that $\stor_j^R(S,M) = 0$ for all $j>0$. But this follows from the Kunneth formula, and  the isomorphisms \eqref{identity:idshf}	- \eqref{identity:schub} together with \Cref{prop:torvan}. 
	
	By \Cref{prop:torvan} and isomorphism \eqref{isom:Tor} applied to $R = \ss{\bar{X}_P},\ S = \ss{X_w^P},\ M = \xi^u_P,$ and $N = \ss{X_x^P}$ (for $x\leq w, x\in W^P$), we have 
	\begin{equation}\label{storvan} \stor_j^{\ss{X_w^P}}(\ss{X_w^P}\otimes_{\ss{\bar{X}_P}}\xi^u_P,\ss{X_x^P}) = 0\hspace{1cm} \text{for all}\ x\leq w,\ j>0.
		\end{equation}
	
	By \Cref{fiber}, the closures of the $\Gamma$-orbits in $(Y_w^P)_\bb{P}$ are precisely  $(X_x^P\times X_y^P)_\bb{P}$ for $x,y\leq w$ and $x,y\in W^P$. Setting $\mathcal{F}: = \ss{(Y_w^P)_\bb{P}}\otimes_{\ss{\xyp{Y}}}(\pi^*\idshf{\bb{P}_\bj}{\partial\bb{P}_\bj}\otimes (\widetilde{\xi^u_P\boxtimes\xi^v_P}))$, the identities \cref{identity:idshf,identity:tilde,identity:schub} and \eqref{storvan} imply that $\mathcal{F}$ is homologically transverse to the closures of the $\Gamma$-orbits in $(Y_w^P)_\bb{P}$. Then applying \cite[Theorem 2.3]{AGM} (with their $G=\Gamma, X=(Y^P_w)_\bb{P}, \mathcal{E}= \tilde{\Delta}_*\ss{(X_w^P)_\bb{P}}$ and their $\mathcal{F}= \mathcal{F}$), we have the following:
	\begin{equation}\label{bigstorvan}\stor_i^{\ss{(Y_w^P)_\bb{P}}}(\ss{(Y_w^P)_\bb{P}}\otimes_{\ss{\xyp{Y}}} (\pi^*(\idshf{\bb{P}_\bj}{\partial \bb{P}_\bj})\otimes(\widetilde{\xi^u_P\boxtimes\xi^v_P})) ,\gamma_*\tilde{\Delta}_*\ss{(X_w^P)_\bb{P}}) = 0\hspace{1cm} \text{for all}\ i>0.
		\end{equation}
	(We note here that although $\Gamma$ is infinite-dimensional, its action on $(Y_w^P)_\bb{P}$ factors through a finite-dimensional quotient group $\bar{\Gamma}$ of $\Gamma$.)	Finally, observe that $\gamma(\tilde{\Delta}(X_w^P)_\bb{P})\subset (Y_w^P)_\bb{P}$, and thus \cref{identity:stors} and \cref{bigstorvan} imply
	\begin{equation*}\stor_i^{\ss{\xyp{Y}}}(\pi^*(\idshf{\bb{P}_\bj}{\partial \bb{P}_\bj})\otimes(\widetilde{\xi^u_P\boxtimes\xi^v_P}),\gamma_*\tilde{\Delta}_*(\ss{(X_w^P)_\bb{P}}))=0\hspace{1cm} \text{for all}\ i>0.\qedhere
		\end{equation*}
	This proves  \Cref{main}(A). 	
	\end{proof}

\section{Study of the Richardson Varieties}\label{rich}

We continue to assume that $N$ is any fixed positive integer and $\mathbb{P} =(\mathbb{P}^N)^r$ is as in Section \ref{mixing}.

For any $u,v\leq w$ in $W^P$ set $X^{u,v}_w:=X^u_w\times X^v_w$ where $X^u_w:=X^u\cap X_w$. Similarly, we set $X^{u,v}_w(P):=X^u_w(P)\times X^v_w(P)$ where $X^u_w(P):=X^u_P\cap X_w^P$. We also write $X^2_w(P):=X_w^P\times X_w^P$ and from now on we denote by $Y_\bj$ (for any $T$-stable subscheme $Y\subset \bar{X}_P$) the inverse image of $\bb{P}_\bj$ under the standard quotient map $E(T)_\bb{P}\times^T Y\to \bb{P}$. 

The action of $B$ on $X_w^P$ factors through the action of a finite dimensional quotient group $\bar{B}=B_w$ containing the maximal torus $H$. Similarly, the action of $\Gamma$ on $(X_w^2(P))_\bb{P}$ descends to an action of a finite dimensional quotient group $\bar{\Gamma} = \Gamma_w$: \[\Gamma\twoheadrightarrow \bar{\Gamma}=\Gamma_w\twoheadrightarrow GL(N+1)^r.\] Further, we can (and do) take $\bar{\Gamma} = \bar{\Gamma}_0 \rtimes GL(N+1)^r$, where $\bar{\Gamma}_0$ is the group of global sections of the bundle $(\bar{B}^2)_\bb{P}:=E(T)_\bb{P}\times^T \bar{B}^2\to \bb{P}$. 

\begin{lemma}\label{map:m} For any $\bj=(j_1,\cdots, j_r)\in[N]^r$ and $u,v\leq w\in W^P$, the map \[m:\bar{\Gamma}\times(X^{u,v}_w(P))_\bj\to (X^2_w(P))_{\bb{P}}\] is flat, where $m(\gamma,x) = \gamma\cdot\pi_2(x)$ and $\pi_2:(X_w^{u,v}(P))_\bj\to (X^2_w(P))_\bb{P}$ is the map induced from the inclusion $p:X^u_w(P)\times X^v_w(P)\to X_w^2(P)$. Similarly, its restriction $\hat{m}:\bar{\Gamma}\times\partial((X^{u,v}_w(P))_\bj)\to (X^2_w(P))_\bb{P}$ is also flat if the domain of $\hat{m}$ is non-empty, where \[\partial((X^{u,v}_w(P))_\bj):=((\partial X^{u,v}_P)\cap (X^2_w(P)))_\bj\cup (X^{u,v}_w(P))_{\partial\bb{P}_\bj},\]  $(X^{u,v}_w(P))_{\partial\bb{P}_\bj}$ denotes the inverse image of $\partial\bb{P}_\bj$ under the quotient map $E(T)_\bb{P}\times^T (X^{u,v}_w(P))\to \bb{P}$ and $\partial X^{u,v}_P:=(\partial X^u_P\times X^v_P)\cup (X^u_P\times \partial X^v_P)$.
\end{lemma}
\begin{proof} Consider the following diagram:
	\begin{center}	\begin{tikzcd} \bar{\Gamma}_0\times X^{u,v}_w(P)\arrow[r]\arrow[d,"m'"] & \bar{\Gamma}\times (X^{u,v}_w(P))_\bj \arrow[r]\arrow[d,"m"] & \text{GL}(N+1)^r\times \bb{P}_\bj\arrow[d,"m''"]\\
			X^2_w(P)\arrow[r] & (X^2_w(P))_\bb{P}\arrow[r] & \bb{P}
	\end{tikzcd}\end{center}
	Here the two right horizontal maps are locally trivial fibrations with fibers in the leftmost spaces. From this and the fact that $m''$ is smooth (see the proof of \cite[Lemma 6.11]{Ku}), to show that $m$ is flat it is enough to show that $m'$ is a flat morphism. From \Cref{fiber}, it suffices to show that $(\bar{B}^2)\times X^{u,v}_w(P)\to X^2_w(P)$ is flat. This follows from the proof  of \cite[Lemma 6.10]{Ku} for $\bar{X}_P$. 
	
	Observe also that, by the same proof as that of \cite[Lemma 6.10]{Ku},  the map $(\bar{B}^2)\times ((\partial X^{u,v}_P)\cap X^2_w(P))\to X^2_w(P)$ is flat. Hence to show that $\hat{m}$ is flat, we first observe that the restrictions to the components $\Gamma_1:=\bar{\Gamma}\times((\partial X^{u,v}_P)\cap X^2_w(P))_\bj$, $\Gamma_2:=\bar{\Gamma}\times (X^{u,v}_w(P))_{\partial\bb{P}_\bj}$, and $\Gamma_1\cap \Gamma_2$ are all flat maps (following the same argument as the first part of this proof). Therefore $\hat{m}$ is flat on $\Gamma_1\cup\Gamma_2$, since for any affine scheme $Y=Y_1\cup Y_2$ with closed subschemes $Y_1,Y_2$ and a morphism $f:Y\to X$ of schemes, there is an exact sequence of $k[X]$-modules: 
	\[0\to k[Y]\to k[Y_1]\oplus k[Y_2]\to k[Y_1\cap Y_2]\to 0.\qedhere\]
\end{proof}

As in Section \ref{notation}, for $v\leq w\in W^P$, let 
$$\pi^v_w: Z^v_w\to X^v_w(P) $$
be the $T$-equivariant desingularization of $X^v_w(P)$. For $u,v \leq  w\in W^P$, let $Z^{u,v}_{w}:=Z^{u}_{w} \times Z^{v}_{w}$ under the
diagonal action of $T$ and  $\left(Z^{u,v}_{w}\right)_{\bj}$ is the inverse image
of $\bb{P}_{\bj}$ under the map $E(T)_\bb{P}
\displaystyle\mathop{\times}^{T} Z^{u,v}_w \to \bb{P}$. 

\vskip1ex

We  record the following from \cite[Lemma 6.11]{Ku}:

\begin{lemma} For any $\bj\in[N]^r$ and $u,v\leq w\in W^P$, the map \[\tilde{m}:\bar{\Gamma}\times (Z^{u,v}_w)_\bj\to (Z^2_w)_\bb{P}\] is a smooth morphism, where $\tilde{m}$ is defined similarly to the map $m$ in \Cref{map:m}. 
\end{lemma}

\begin{lemma} \label{prop:norm} For any $u\in W^P$, $X^u_P$ is normal and Cohen-Macaulay.
	\end{lemma}
\begin{proof} The standard projection map $p:G/B\to G/P$ is a locally trivial fibration with fibers isomorphic to the finite dimensional flag variety $P/B$ (since, by assumption, $P$ is of finite type). By \cite[Propositions 3.2 and 3.4]{KS} the lemma follows for $X^u$. Hence, the lemma follows for $X^u_P$ since $X^u\to X^u_P$ is a locally trivial fibration with fiber $P/B$. Thus, locally $X^u$ is isomorphic with $X^u_P\times \mathbb{A}^n$, for $n=\dim P/B$. See \cite[Lemma 10.164.3]{stack} for normality and  \cite[Ch. II, Theorem 8.21A.(d)]{H} for the Cohen-Macaulay property. 
	\end{proof}

\begin{lemma}\label{rich:props} For any symmetrizable Kac-Moody  group $G$ and any $u\leq w$ in $W^P$, the Richardson variety $X^u_w(P)\subset \bar{X}_P$ is irreducible, normal, and Cohen-Macaulay with rational singularities. Moreover, $C_w^P\cap C^u_P$ is an open dense subset of $X^u_w(P)$. 
\end{lemma}
\begin{proof} Observe that under the locally trivial fibration $\pi: \bar{X} \to \bar{X}_P$ with fiber $P/B$, $\pi^{-1}(X^u_w(P))= X^u_{\bar{w}}$,  where $\bar{w}$ is the longest element in the coset $wW_P$.  Hence, the irreducibility, normality and Cohen-Macaulay property of $X^u_w(P)$ follows,  as in the above lemma \ref{prop:norm},  from the corresponding properties of $X^u_{\bar{w}}$ (cf.   \cite[Propositon 6.6]{Ku}). Similarly, the rational singularities of $X^u_w(P)$ follows from the corresponding result for $X^u_{\bar{w}}$ (cf.  \cite[Theorem 3.1
and Remark 2.2]{KuS}). 

Finally, $C_w^P\cap C^u_P$ is clearly an open subset of $X^u_w(P)$ and therefore dense since it is nonempty by \cite[Lemma 7.3.10]{K}. 
\end{proof}

\begin{lemma}\label{CM:bdy} For $u \in W^P$, $\mathcal{O}_{X^u_P}(-\partial X^u_P)$ is a Cohen-Macaulay $\mathcal{O}_{X^u_P}$-module. Hence, $\partial X^u_P$ is a Cohen-Macaulay variety.

Thus, for any $w\in W^P$, if non-empty, $X_w^P\cap (\partial X^u_P)$ is a Cohen-Macaulay variety.
	\end{lemma}

\begin{proof} The argument is an adaptation of \cite[Proof of Lemma 4]{Br}. In the following, even though we deal with various schemes of infinite type over $\mathbb{C}$, the assertions reduce to schemes of finite type by quotienting these schemes by a closed normal subgroup $N$ of $B^-$ of finite codimension (cf. \cite[Lemma 6.1]{Ku}).
	 
	 For $u\in W^P$, let $\bar{u}$ be the longest element in the coset $uW_P$.
Fix any finite ideal $S\subset W$ containing $\bar{u}$ such that $S$ is stable under the right multiplication by $W_P$ and let $V^S$ be the corresponding open set in $\bar{X}$ as in \cite[Definition 3.1]{Ku} and let 
$V^S_P$ be its image in $\bar{X}_P$. To prove the proposition, we can replace $X^u_P$ by its open subset $X^u_P(S):= X^u_P\cap V^S_P$.
 Consider the morphisms 
$$Z^{\bar{u}} (S)\xrightarrow{\psi} X^{\bar{u}}(S)\xrightarrow{\eta} X^u_P(S),$$ 
where  $X^{\bar{u}}(S):= X^{\bar{u}}\cap V^S$, 
$\psi$ is the desingularization as in \cite[Theorem 6.4]{Ku} and $\eta$ is the standard projection map, which is birational. Let $$\bar{\eta}:Z^{\bar{u}} (S)\rightarrow X^u_P(S)$$ be the composite map $\eta\circ \psi$.	
Since $X^u_P$ has rational singularities by the next Lemma \ref{lemma0.2}, $$\eta_*(\mathcal{O}_{X^{\bar{u}}(S)})=\mathcal{O}_{X^u_P(S)}\ \text{ and }\ R^i\eta_*(\mathcal{O}_{X^{\bar{u}}(S)})=0, \text{ for all } i>0.$$ It is easy to see that $\eta^{-1}(\partial X^u_P(S)) = \partial X^{\bar{u}}(S)$,
where $\partial X^{\bar{u}}(S):= (\partial X^{\bar{u}})\cap V^S$ etc. Thus, 
\begin{equation}\label{etapush} {\eta}_*(\mathcal{O}_{X^{\bar{u}}(S)}(-\partial X^{\bar{u}}(S))=\mathcal{O}_{X^u_P(S)}(-\partial X^u_P(S)).
	\end{equation}
	From the Grauert-Riemenschneider Theorem 
	$$ R^i\bar{\eta}_*(\omega_{Z^{\bar{u}}(S)})=0 \text{ and } R^i\psi_*(\omega_{Z^{\bar{u}}(S)})=0, \text{ for } i>0,
	$$
	 where $\omega$ is the canonical sheaf. Thus,
	\begin{equation} \label{eqn100} R^i\psi_*\left(\omega_{Z^{\bar{u}}(S)}\otimes \psi^*(\mathcal{L}(\rho))\right)\simeq R^i\psi_*\left(\omega_{Z^{\bar{u}}(S)}\right)\otimes \mathcal{L}(\rho)	
	=0, \text{  } i>0.
	\end{equation}	
	By the Relative Kawamata-Viehweg Vanishing  Theorem \cite[Theorem 8.3]{Ku} (since $\mathcal{L}(\rho)$ is ample on $X^{\bar{u}}(S)$
	and hence $\psi^*(\mathcal{L}(\rho))$ is nef and big over $Z^{\bar{u}}(S)$)
	\begin{equation} \label{eqn101} 		R^i\bar{\eta}_*\left(\psi^*(\mathcal{L}(\rho))\otimes \omega_{Z^{\bar{u}}(S)}\right)=0, \,\,\,\text{ for } i>0.	\end{equation}		
	By the Grothendieck spectral sequence \cite[Part I, \S 4.1]{J} for the composite map $\bar{\eta} = \eta\circ\psi$, using $\psi_*(\omega_{Z^{\bar{u}}(S)}) = \omega_{X^{\bar{u}}(S)}$ as in \cite[Theorem 5.10]{KM}(since $X^{\bar{u}}$ has rational singularities
	by Lemma \ref{lemma0.2})  and using the identities \eqref{eqn101} and \eqref{eqn100},	
	we get 
	\begin{equation} \label{eqn102} R^i\eta_*\left(\mathcal{O}_{X^{\bar{u}}(S)}(-\partial X^{\bar{u}}(S)\right)\simeq	R^i\eta_*\left(\mathcal{L}(\rho)\otimes \omega_{X^{\bar{u}}(S)}\right)=0 , \,\,\,\text{ for } i>0.	\end{equation}	
	where the first isomorphism follows from \cite[Theorem 10.4]{Ku}. 
	
	Take the ample line bundle $\mathcal{M}':=\mathcal{L}^P(\hat{\rho}_Y)$ over ${X}_P^u$. It has a section with support precisely equal to $\partial X^u_P$  (cf. Proof of Lemma \ref{sec:support}). Let $\mathcal{M}$ be its pull-back to $Z^{\bar{u}} (S)$ via $\bar{\eta}$. Then, $\mathcal{M}$ has support precisely equal to  $\partial Z^{\bar{u}}(S) := \bar{\eta}^{-1} (\partial X^u_P(S))$ with divisor $\sum_{i=1}^d b_i Z_i$ (where $b_i>0$ for all $i$), $Z_i$ being the irreducible components of  	$\partial Z^{\bar{u}}(S)$.	(Since $\partial Z^{\bar{u}}(S)$ is the zero set of a line bundle on $Z^{\bar{u}}(S)$, $\partial Z^{\bar{u}}(S)$ is a pure scheme of codimension 1 in $Z^{\bar{u}}(S)$.) 	Further, $\mathcal{M}$ is nef and big (since $\bar{\eta}$ is birational and $\mathcal{M}'$ is ample). Observe that $\mathcal{M}'$ descends to an ample  line bundle over $N\backslash X^u_P(S)$ and the section $\chi(e^*_{\bar{u}\hat{\rho}_Y})$ descends to a section over $N\backslash X^u_P(S)$.	
	
		Let $\mathcal{L}$ be the line bundle on the smooth scheme $Z^{\bar{u}}(S)$ associated to the reduced divisor $\partial Z^{\bar{u}}(S)$ and let $D$ be the divisor $\sum_i (N-b_i) Z_i$ on $Z^{\bar{u}}(S)$, where we choose $N>b_i$ for all $i$.	As  in \cite[Theorem 6.4]{Ku},  $\partial Z^{\bar{u}}(S) = \sum Z_i$ is a simple normal crossing divisor.	
	 Clearly, \[\mathcal{L}^N(-D) =\mathcal{O}_{Z^{\bar{u}}(S)}\left(\sum_i b_iZ_i\right).\]
	 
	 Thus, the vanishing 
	 \begin{equation} \label{eqn103} R^i\bar{\eta}_*\left( \omega_{Z^{\bar{u}}(S)}(\partial Z^{\bar{u}}(S)\right)=0,\,\,\text{ for $i>0$},
	 \end{equation} 	 
	follows from \cite[Theorem 8.3]{Ku}. 	 
	
	By a similar argument applied to $\psi$ and taking the line bundle $\mathcal{M}'=\mathcal{L}(\rho)$ over $X^{\bar{u}}$, we get
	 \begin{equation} \label{eqn104} R^i{\psi}_*\left( \omega_{Z^{\bar{u}}(S)}(\partial Z^{\bar{u}}(S)\right)=0,\,\,\text{ for $i>0$}.
	 \end{equation} 	 	
	Further,
	\begin{equation} \label{eqn105}\psi_*\left( \omega_{Z^{\bar{u}}(S)}(\partial Z^{\bar{u}}(S))\right)\simeq  \omega_{X^{\bar{u}}(S)}(\partial X^{\bar{u}}(S))
	\end{equation}
	by the same proof as that of \cite[Theorem 8.5 (b)]{Ku}. 
	Thus, by the Grothendieck spectral sequence for the composite map $\bar{\eta} = \eta\circ\psi$, and using the equations 
	\eqref{eqn103}, \eqref{eqn104} and \eqref{eqn105}, we get:
	 \begin{equation} \label{eqn106} R^i{\eta}_*\left( \omega_{X^{\bar{u}}(S)}(\partial X^{\bar{u}}(S)\right)=0,\,\,\text{ for $i>0$}.
	 \end{equation} 	 	
		Further, 
	\begin{align*} {\Ext}^i_{\mathcal{O}_{X^P(S)}}(\mathcal{O}_{X_P^u(S)} (-\partial X_P^u(S)),\omega_{X^P(S)}) & \approx {\Ext}^i_{\mathcal{O}_{X^P(S)}}\left(\eta_*(\mathcal{O}_{X^{\bar{u}}(S)}(-\partial X^{\bar{u}} (S))),\omega_{X^P(S)}\right), \text{ by \eqref{etapush}}\\
		& \approx  R^{i-l(u)}\eta_*\left(\omega_{X^{\bar{u}}(S)}(\partial X^{\bar{u}}(S))\right), \text{ by the duality for  $\eta$ using }\eqref{eqn102}\\
		& =  0 \text{ if } i\neq l(u), \,\,\text{using} \, \eqref{eqn106}.
		\end{align*}
		Thus, $${\Ext}^i_{\mathcal{O}_{X^P(S)}}(\mathcal{O}_{X_P^u(S)}(-\partial X_P^u(S)),\omega_{X^P(S)})=0,\text{ for } i\neq l(u).$$ This proves that $\mathcal{O}_{X_P^u(S)}(-\partial X_P^u(S))$ is a Cohen-Macaulay $\mathcal{O}_{X_P^u(S)}$-module. From this we easily see that $\partial X_P^u(S)$ is a Cohen-Macaulay variety (following the same argument as in \cite[Proof of Corollary 10.5]{Ku}). 
		
To prove that $X_w^P\cap (\partial X^{{u}}_P)$ is a Cohen-Macaulay variety, follow exactly the same argument as in the proof of \cite[Proposition 6.6]{Ku} by taking the maps (with the notation in loc cit.)
$$\mu: G\times^B X_w^P\to \bar{X}_P, \,\,[g, x] \mapsto g\cdot x,$$
and
$$\pi: G\times^B X_w^P\to \bar{X}, \,\,[g, x] \mapsto gB.$$
Then, $\mu$ is a fibration with fiber $X^B_{\bar{w}^{-1}}$. Rest of the argument is the same.
		
	\end{proof}

\begin{lemma} \label{lemma0.2} For any $u\in W^P$, $X^u_P$ has rational singularities.
\end{lemma}

\begin{proof} It suffices to prove that $X^u$ has rational singularities for any $u\in W$. Consider the boundary as a (reduced) divisor 
	$D=\sum_i Z_i$. By  \cite[Theorem 10.4]{Ku}	and
	\cite[Proof of Proposition 5.3, Proposition 5.5 and Theorem 5.7]{KuS}, $(X^u, D)$ is log canonical. Consider the line bundle $\mathcal{L}(\rho)_{|X^u}$. By \cite[Proof of Lemma 8.1]{Ku}, it has a section with divisor $\sum_i b_iZ_i$, for some $b_i>0$. Take a positive integer $N>b_i$ for all $i$ and consider the $\mathbb{Q}$-divisor $\Delta= \sum_i (1-\frac{b_i}{N})Z_i$. Then, 
	$$\omega_{X^u} +\Delta = -(1+\frac{1}{N})\sum_i b_iZ_i ,$$
	and hence it is a $\mathbb{Q}$-Cartier divisor. Thus, by \cite[Lemma 5.9]{KuS}, $(X^u, \Delta)$ is KLT. In particular, by \cite[Remark 2.2]{KuS}, $X^u$ has rational singularities. This proves the lemma.
\end{proof}

\section{The schemes $\mathcal{Z}_P$ and $\widetilde{\mathcal{Z}}$}

We continue to assume that $N$ is any fixed positive integer and $\mathbb{P} =(\mathbb{P}^N)^r$ is as in Section \ref{mixing}.

Let $u,v\leq w\in W^P$. As in Section \ref{rich}, denote by $Z^{u,v}_w:=Z^u_w\times Z^v_w$, where $Z^u_w$ is the $T$-equivariant desingularization of $X^u_w(P)$ as in Section \ref{notation}. We also let $Z^2_w:=Z_w\times Z_w$, where $Z_w$ is a BSDH variety as in \cite[\S 7.1.3]{K}. For any $\bj\in [N]^r$, let $(X^{u,v}_w(P))_\bj$ and $(Z^{u,v}_w)_\bj$ denote the inverse image of $\bb{P}_\bj$ through the maps $E(T)_\bb{P}\times^T X^{u,v}_w(P)\to\bb{P}$ and $E(T)_\bb{P}\times^T Z^{u,v}_w\to\bb{P}$ respectively. 

We define the scheme $\widetilde{\mathcal{Z}}$ to be the fiber product $(\bar{\Gamma}\times (Z^{u,v}_w)_\bj)\times_{(Z^2_w)_\bb{P}} \widetilde{\Delta}((Z_w)_\bb{P})$ and $\mathcal{Z}_P$ to be the fiber product $(\bar{\Gamma}\times (X^{u,v}_w(P))_\bj)\times_{(X^2_w(P))_\bb{P}} \widetilde{\Delta}((X^P_w)_\bb{P})$ as in the commutative diagram:\\

\begin{center}
\begin{tikzcd} & \widetilde{\mathcal{Z}}\arrow[rr, "\tilde{\mu}","\text{(smooth)}"']\arrow[ddd,"\tilde{i}",hook]\arrow[ldddd,"\widetilde{\pi}"']\arrow[dddddddd,bend right=90,"f"']\arrow[dddrr,phantom,"\square"] & & \widetilde{\Delta}((Z_w)_\bb{P})\arrow[ddd,hook]\\
	& & & \\
	& & & \\
	& \bar{\Gamma}\times (Z^{u,v}_w)_\bj \arrow[rr,"\tilde{m}","\text{(smooth)}"'] \arrow[dd,"\theta"] & & (Z^2_w)_\bb{P}\arrow[dd,"\beta"]\\
	\bar{\Gamma}& & & \\
	& \bar{\Gamma}\times (X^{u,v}_w(P))_\bj \arrow[rr,"m","\text{(flat)}"' 	]\arrow[dddrr,phantom,"\square"] & & (X^2_w(P))_\bb{P}\\
	& & & \\
	& & & \\
	& \arrow[uuuul,"\pi"]\mathcal{Z}_P\arrow[uuu,hook,"i"'] \arrow[rr,"\mu","\text{(flat)}"'] & & \widetilde{\Delta}((X_w^P)_\bb{P})\arrow[uuu,hook]	
	\end{tikzcd}
\end{center}

The map $\theta$ is induced by the product of the desingularizations $\pi^u_w:Z^{u}_w\to X^{u}_w(P)$  and $\pi^v_w:Z^{v}_w\to X^{v}_w(P)$. In particular, $\theta$ is a birational map for $u,v\leq w\in W^P$.  The map $\beta$ is induced from the BSDH desingularization $Z_w\to X_w^P$. The maps $\pi$ and $\widetilde{\pi}$ are obtained from the projections to the $\bar{\Gamma}$-factor via the maps $i$ and $\tilde{i}$ respectively. The map $f$ is the restriction of $\theta$ to $\widetilde{\mathcal{Z}}$ (via $\tilde{i}$) with image inside $\mathcal{Z}_P$. 

\begin{proposition} \label{prop8.1}The schemes $\mathcal{Z}_P$ and $\widetilde{\mathcal{Z}}$ are irreducible and the map $f:\widetilde{\mathcal{Z}}\to \mathcal{Z}_P$ is a proper birational map. Thus $\widetilde{\mathcal{Z}}$ is a desingularization of $\mathcal{Z}_P$. Moreover, $\mathcal{Z}_P$ is Cohen-Macaulay with 
	\begin{equation}\label{dimZ} \dim(\mathcal{Z}_P)=|\bj|+l(w)-(l(u)+l(v))+\dim(\bar{\Gamma}),
		\end{equation}
	where $|\bj|=\sum_{i=1}^r j_i$ for $\bj=(j_1,\cdots,j_r)$. 
	\end{proposition}

\begin{proof} We first note that the top Cartesian square is precisely the same as that of \cite[\S 7]{Ku}, and hence (as in \cite[Equation 59]{Ku})
\begin{equation} \label{dim} \dim(\widetilde{\mathcal{Z}}) = \dim(\bar{\Gamma}) +|\bj|+l(w) -l(u) -l(v)
\end{equation} and $\widetilde{\mathcal{Z}}$ is irreducible from  \cite[Proposition 7.4]{Ku}. 

We next show that 	$\mathcal{Z}_P$ is of pure dimension. Since $m$ is a flat morphism, Im $(m)$ is an open subset of $(X^2_w(P))_\bb{P}$ (\cite[Ch. III, Exercise 9.1]{H}). Moreover, clearly Im $(m) \supset (C^2_w(P))_\bb{P}$, thus Im $(m)$ intersects $\tilde{\Delta}((X_w(P)_\bb{P})$. Applying \cite[Ch. III, Corollary 9.6]{H} first to the morphism $m: \bar{\Gamma}\times (X^{u,v}_w(P))_\bj \to$ Im $(m)$ and then its restriction $\mu$ to $\mathcal{Z}_P$, we see that $\mathcal{Z}_P$ is pure dimensional. 

		We next show that $\mathcal{Z}_P$ is irreducible. Since 	$\widetilde{\mathcal{Z}}$ is irreducible, so is its open subset $\widetilde{\mathcal{Z}}\cap \widetilde{m}^{-1}((C^2_w(P))_\bb{P})$. Now $\widetilde{\mathcal{Z}}\cap \widetilde{m}^{-1}((C^2_w(P))_\bb{P})$ maps surjectively onto the open subset $\mathcal{Z}_P\cap m^{-1}((C^2_w(P))_\bb{P})$ of $\mathcal{Z}_P$	
	under $f$ and hence $\mathcal{Z}_1: = \overline{\mathcal{Z}_P\cap m^{-1}((C^2_w(P))_\bb{P})}$ is an irreducible component of $\mathcal{Z}_P$. Let $\mathcal{Z}_2$ be another irreducible component of $\mathcal{Z}_P$ so that $\mu(\mathcal{Z}_2)\subset \tilde{\Delta}((X^P_w\setminus C^P_w)_\bb{P})$. But $\dim(\tilde{\Delta}((X^P_w\setminus C^P_w)_\bb{P}))<\dim (\tilde{\Delta}((X^P_w)_\bb{P}))$ and all of the fibers of $\mu\vert_{\mathcal{Z}_2}$ have dimension no more than any fiber of $\mu$, so $\dim(\mathcal{Z}_2)<\dim(\mathcal{Z}_1)$, which is a contradiction since $\mathcal{Z}_P$ is of pure dimension. Thus $\mathcal{Z}_P = \mathcal{Z}_1$ is irreducible.
	
	 The map $f$ is clearly proper and an isomorphism onto an open subset of $\mathcal{Z}_P$	 
	 when restricted to the open subset \[\widetilde{\mathcal{Z}}\cap(\bar{\Gamma}\times((C^u\cap C_w)\times (C^v\cap C_w))_\bj),\] where we identify the inverse image of $(C^u\cap C_w) \subset X^u_w$ inside $Z^u_w$ with $C^u\cap C_w$. Hence $\dim(\mathcal{Z}_P)=\dim(\widetilde{\mathcal{Z}})$ and \eqref{dimZ} holds by the equation \eqref{dim}. Therefore, \[\text{codim}_{\mathcal{Z}_P}(\bar{\Gamma}\times (X^{u,v}_w(P))_\bj) = \text{codim}_{\tilde{\Delta}((X^P_w)_\bb{P})}((X^2_w(P))_\bb{P})=l(w).\]
	 The fact that $\mathcal{Z}_P$ is Cohen-Macaulay follows from the fact that $m$ is flat as well as \cite[Lemma 7.2]{Ku} and \Cref{rich:props}.
	\end{proof}

\begin{lemma} The scheme $\mathcal{Z}_P$ is normal, irreducible, and Cohen-Macaulay.
	\end{lemma}
\begin{proof} The only thing that remains to be shown is that $\mathcal{Z}_P$ is normal. Consider the map \[\mu_o: G\times^{U^-} X^u_P(S_u')\to \bar{X}_P,\hspace{1cm} [g,x]\mapsto g\cdot x,\] where $X^u_P(S_u'):=X^u_P\cap V^{S_u'}$, $V^{S_u'}:=\bigcup_{w\in S_u'} B^-wx_o$, $x_o$ is the base point $1.P$ of $\bar{X}_P$, and $S_u':=\{v\in W^P\mid l(v)\leq l(u)+1\}$. Since $\mu_o$ is $G$-equivariant, it is a locally trivial fibration with fibers $F^u:=\bigcup_{u\leq v;\ l(v)\leq l(u)+1} Pv^{-1}U^-/U^-$. Since $X^u_P$ is normal by \Cref{prop:norm} and any $B^-$ orbit in $X^u_P(S_u')$ has codimension $\leq 1$ in $X^u_P$, $X^u_P(S_u')$ is smooth and so is $F^u$. (Here
the smoothness of $F^u$ means that there exists a closed normal subgroup $B_1$ of
$B$ of finite codimension such that $B_1$ acts freely and properly on $F^u$
 and the quotient $B_1\setminus F^u$ is a
smooth scheme of finite type over $\mathbb{C}$.) 
Therefore $\mu_o$ is a smooth morphism. Hence so is the restriction of $\mu_o$ to the open subset $B\times X^u_P(S_u')$, and further so is the restriction to the inverse image of $X_w^P$, $\mu_o(w):B\times (X^u_P(S_u')\cap X_w^P)\to X_w^P$. Clearly $\mu_o(w)$ factors through a smooth morphism $\bar{\mu}_o(w):\bar{B}\times (X^u_P(S_u')\cap X_w^P)\to X_w^P$, where $\bar{B}$ is a finite-dimensional quotient of $B$. Hence, $\bar{\mu}_o(w)^{-1}(\mathring{X}_w^P) = \bar{B}\times (X^u_P(S_u')\cap (\mathring{X}_w^P))$ is a smooth variety, where $\mathring{X}_w^P:=X_w^P\setminus \Sigma_w^P$ and $\Sigma_w^P$ is the singular locus of $X_w^P$.

Following the same argument as in the proof of  \Cref{map:m} and restricting the middle vertical map to the open subset $\bar{m}:\bar{\Gamma}\times ((X^u_P(S_u')\times X^v_P(S_v'))\cap X^2_w(P))_\bj\to (X^2_w(P))_\bb{P}$ and the left vertical map accordingly, we have that $\bar{m}$ is a smooth morphism (with open image $Y$). Hence, the restriction of $\bar{m}$ to $\tilde{m}:\bar{m}^{-1}(\tilde{\Delta}((X_w^P)_\bb{P}))\to\tilde{\Delta}((X_w^P)_\bb{P})$ is also smooth. (Observe that $Y$ does intersect $\tilde{\Delta}((X_w^P)_\bb{P})$ since 
$Y$ being a $\bar{\Gamma}$-stable open subset of  $(X^2_w(P))_\bb{P}$,  $(C^2_w(P))_\bb{P}\subset Y$.)
Therefore the open set $\tilde{m}^{-1}(\tilde{\Delta}((\mathring{X}_w^P)_\bb{P}))$ in $\mathcal{Z}_P=m^{-1}(\tilde{\Delta}((X_w^P)_\bb{P}))$ is a smooth variety. Denote the complement of $\bar{\Gamma}\times ((X^u_P(S_u')\times X^v_P(S_v'))\cap X^2_w(P))_\bj$ in $\bar{\Gamma}\times (X^{u,v}_w(P))_\bj$ by $F$ and denote $\tilde{m}^{-1}(\tilde{\Delta}((\Sigma_w^P)_\bb{P}))$ by $F'$. Then $F'$ has codimension $\geq 2$ in $\bar{m}^{-1}(\tilde{\Delta}((X_w^P)_\bb{P}))$, and hence in $\mathcal{Z}_P$. Also, $F$ is of codimension $\geq 2$ in $\bar{\Gamma}\times(X^{u,v}_w(P))_\bj$. If $F$ is nonempty, then the restriction of $m$ to $F$ is flat (following the proof of \Cref{map:m}) with image an open subset of $(X^2_w(P))_\bb{P}$ intersecting $\tilde{\Delta}((X^P_w)_\bb{P})$. Therefore the codimension of $F\cap \mathcal{Z}_P$ in $\mathcal{Z}_P$ is $\geq 2$ and hence so is the complement of the smooth locus of $\mathcal{Z}_P$
 in $\mathcal{Z}_P$. Finally, since $\mathcal{Z}_P$ is Cohen-Macaulay, it is normal by Serre's criterion \cite[Ch. II, Theorem 8.22(A)]{H}.
	\end{proof}

\begin{proposition}\label{ratsings} The scheme $\mathcal{Z}_P$ has rational singularities. 
	\end{proposition}
\begin{proof} Since $\mu$ is flat and $\tilde{\Delta}((X_w^P)_\bb{P})$ has rational singularities \cite[Theorem 8.2.2(c)]{K}, it is enough to show (by \cite[Th\'{e}or\`{e}me 5]{El}) that the fibers of $\mu$ are disjoint unions of irreducible varieties with rational singularities. 
	
	Let $x\in \tilde{\Delta}((C_{w'}^P)_\bb{P})$, where $w'\leq w \in W^P$. Then from \cite[Lemma 7.6]{Ku} and \Cref{fiber}, 
	\begin{equation} \label{eqn8.3} \text{Stab}(x)\setminus \mu^{-1}(x)\simeq ((X^u_P\cap C_{w'}^P)\times(X^v_P\cap C^P_{w'}))_\bj,
	\end{equation}
	 where the stabilizer Stab$(x)$ is taken with respect to the $\bar{\Gamma}$ action on $(X^2_w(P))_\bb{P}$. From \cite[Proposition 3, \S 2.5]{Se}, the map $\bar{\Gamma}\to \text{Stab}(x)\setminus \bar{\Gamma}$ is locally trivial in the \'{e}tale topology. 
	
	We have the pull-back diagram 
	\begin{center}\begin{tikzcd} \mu^{-1}(x)\arrow[d] & \subseteq & \bar{\Gamma}\times (X^{u,v}_w(P))_\bj\arrow[d]\\ \text{Stab}(x)\setminus\mu^{-1}(x) & \subseteq & (\text{Stab}(x)\setminus\bar{\Gamma})\times (X^{u,v}_w(P))_\bj\,,
			\end{tikzcd}
		\end{center}
	where the left vertical map is a locally trivial fibration in the \'{e}tale topology since so is the right side map. Therefore, since $\text{Stab}(x)$ is smooth and $\text{Stab}(x)\setminus\mu^{-1}(x)$ has rational singularities (by the equation \eqref{eqn8.3} and Lemma \ref{rich:props}), $\mu^{-1}(x)$ is a disjoint union of irreducible varieties with rational singularities by \cite[Corollary 5.11]{KM}.
	\end{proof}

\begin{proposition}\label{prop:bdyZ} If non-empty, the scheme $\partial \mathcal{Z}_P$ is of pure codimension 1 in $\mathcal{Z}_P$ and is Cohen-Macaulay, where we define the closed subscheme $\partial \mathcal{Z}_P$ of $\mathcal{Z}_P$ as \begin{equation*} \partial\mathcal{Z}_P:=(\bar{\Gamma}\times\partial ((X^{u,v}_w(P))_\bj))\times_{(X^2_w(P))_\bb{P}} \widetilde{\Delta}((X_w^P)_\bb{P}),
	\end{equation*}
where $\partial ((X^{u,v}_w(P))_\bj)$ is as in Lemma \ref{map:m}.
	\end{proposition}

\begin{proof} First recall that $\bar{\Gamma}\times \partial((X^{u,v}_w(P))_\bj)\xrightarrow{\hat{m}} (X^2_w(P))_\bb{P}$ is a flat morphism (cf. \Cref{map:m}) and observe that $\partial ((X^{u,v}_w(P))_\bj)$ is pure of codimension 1 in $(X^{u,v}_w(P))_\bj$. We consider two cases: $\partial\bb{P}_\bj\neq\emptyset$ or $\partial \bb{P}_\bj=\emptyset$. In the first case, $\text{Im }\hat{m}=\text{Im }m$. In the latter case, we have \[\text{Im }\hat{m}\supset \left(\left(\left(\bigcup_{u\to u'\leq \theta\leq w}C_\theta^P\right)\times\left(\bigcup_{v\leq \theta'\leq w}C_{\theta'}^P\right)\right)\cup\left(\left(\bigcup_{u\leq \theta\leq w}C_\theta^P\right)\times\left(\bigcup_{v\to v'\leq \theta'\leq w}C_{\theta'}^P\right)\right)\right)_\bb{P}.\] In either case we have that, if nonempty, $\text{Im }\hat{m}$ intersects $\tilde{\Delta}((X_w^P)_\bb{P})$. Moreover, since $\hat{m}$ is flat, $\text{Im }\hat{m}$ is open in $(X^2_w(P))_\bb{P}$. Thus, by \cite[Ch. III, Corollary 9.6]{H}, each fiber of $\hat{m}$ (if nonempty) is pure of dimension \[\dim(\bar{\Gamma})+\dim((X^{u,v}_w(P))_\bj)-\dim((X^2_w(P))_\bb{P})-1.\]
	
	Applying \cite[Ch. III, Corollary 9.6]{H} again to the restriction of $\hat{m}$ to $\partial\mathcal{Z}_P$ via $\mu$, we have that $\partial\mathcal{Z}_P$ is pure of dimension \[\dim(\bar{\Gamma})+\dim((X^{u,v}_w(P))_\bj)-\dim((X^2_w(P))_\bb{P})-1+\dim(\tilde{\Delta}((X_w^P)_\bb{P})).\] Hence, from 
	\eqref{dimZ} it is clear that $\partial\mathcal{Z}_P$ is pure of codimension 1 in $\mathcal{Z}_P$. From Lemma  \ref{rich:props}	 and 	
	  \Cref{CM:bdy}, both of $((\partial X^u_P)\cap X_w^P)\times X^v_w(P)$ and $X^u_w(P)\times ((\partial X^v_P)\cap X_w^P)$ are Cohen-Macaulay and  so is their intersection, which is of pure codimension 1 in each of them. Thus their union is Cohen-Macaulay (cf. \cite[Theorem A.36]{K}) and hence so is $((\partial X^{u,v}_P)\cap X^2_w(P))_\bj$. We also have that $(X^{u,v}_w(P))_{\partial\bb{P}_\bj}$ and the intersection \[((\partial X^{u,v}_P)\cap X^2_w(P))_\bj\cap (X^{u,v}_w(P))_{\partial\bb{P}_\bj}=((\partial X^{u,v}_P\cap X^2_w(P))_{\partial\bb{P}_\bj}\] are Cohen-Macaulay since $\partial \bb{P}_\bj$ is. Since $((\partial X^{u,v}_P)\cap X^2_w(P))_{\partial\bb{P}_\bj}$ is Cohen-Macaulay of pure codimension 1 in each of $(X^{u,v}_w(P))_{\partial\bb{P}_\bj}$ and $((\partial X^{u,v}_P)\cap X^2_w(P))_\bj$, their union $\partial((X^{u,v}_w(P))_\bj)$ is Cohen-Macaulay. Finally, applying \cite[Lemma 7.2]{Ku} to $\hat{m}$ shows $\partial \mathcal{Z}_P$ is Cohen-Macaulay. 
	\end{proof}

\begin{corollary} \label{coro8.5} Assume that $c_{u,v}^w(\bj)\neq 0$, where  $c_{u,v}^w(\bj)$ is defined by the equation \eqref{equivcoeff:c} (also see Lemma \ref{lemma:c}). 
 Then, $\pi$ is surjective and  for general $\gamma\in \bar{\Gamma}$, the fiber $N_\gamma:=\pi^{-1}(\gamma)\subset\mathcal{Z}_P$ is Cohen-Macaulay of pure dimension, where $\pi:\mathcal{Z}_P\to \bar{\Gamma}$ is defined at the beginning of this section. In fact, for any $\gamma\in\bar{\Gamma}$ such that $N_\gamma$ is pure of dimension  \begin{equation} \dim(N_\gamma) = \dim(\mathcal{Z}_P)-\dim(\bar{\Gamma}) = |\bj|+l(w)-l(u)-l(v),
		\end{equation}
	$N_\gamma$ is Cohen-Macaulay (and this condition holds for general $\gamma$). 
	
	Similarly, if $|\bj|+l(w)-l(u)-l(v)>0$, then for general $\gamma\in\bar{\Gamma}$, the fiber $M_\gamma:=\pi_1^{-1}(\gamma)\subset\partial \mathcal{Z}_P$ is Cohen-Macaulay of pure codimension 1 in $N_\gamma$, where $\pi_1$ is the restriction of $\pi$ to $\partial\mathcal{Z}_P$. If $|\bj|+l(w)-l(u)-l(v) = 0$, then for general $\gamma\in\bar{\Gamma}$, $M_\gamma$ is empty. 
	
	In particular, for general $\gamma\in\bar{\Gamma}$, 
	\begin{equation} \label{eqn21} \sext^i_{\ss{N_\gamma}}(\idshf{N_\gamma}{M_\gamma},\omega_{N_\gamma})=0\hspace{2mm}\text{ for all  } i>0,
		\end{equation}
		where $\omega_{N_\gamma}$ is the dualizing sheaf of $N_\gamma$. 		
	\end{corollary}

\begin{proof} This follows the same argument as the proof of \cite[Corollary 7.9]{Ku}.
	\end{proof}

\section{Study of $R^pf_*(\omega_{\widetilde{\mathcal{Z}}}(\partial \widetilde{\mathcal{Z}}))$}

We continue to assume that $N$ is any fixed positive integer and $\mathbb{P} =(\mathbb{P}^N)^r$ is as in Section \ref{mixing}.

Throughout this section we assume $c_{u,v}^w(\bj)\neq 0$, where $c_{u,v}^w(\bj)$ is defined by the identity \eqref{equivcoeff:c}. We also follow the notation of the commutative diagram at the beginning of Section 8. 

\begin{lemma}\label{sec:support} For $u\in W^P$, the line bundle $\mathcal{L}^P(\hat{\rho}_Y)\vert_{X^u_P}$ has a section with zero set precisely equal to $\partial X^u_P$. In particular, \begin{equation*} \mathcal{L}^P(\hat{\rho}_Y)\vert_{X^u_w(P)}\sim \sum_i b_iX_i\hspace{4mm}\text{ for some } b_i>0,
		\end{equation*}
	where the $X_i$ are the irreducible components of $(\partial X^u_P)\cap X_w^P$. 
	\end{lemma}
\begin{proof} Let $L(\hat{\rho}_Y)$ denote the integrable highest weight  $G^\text{min}$-module with highest weight $\hat{\rho}_Y$ and let 
$L(\hat{\rho}_Y)^\vee$ be its restricted dual, i.e., the direct sum of the dual of its weight spaces. Also, let $\mathcal{L}$ be the tautological line bundle over  $\bb{P}(L(\hat{\rho}_Y))$. 
Consider the linear map $\beta:L(\hat{\rho}_Y)^\vee\to H^0(\bb{P}(L(\hat{\rho}_Y)), \mathcal{L}^*)$, where $\beta(f)(x) = (x, f\vert_x)$ for $f\in L(\hat{\rho}_Y)^\vee$ and $x\in \bb{P}(L(\hat{\rho}_Y))$. Further,  let $i^*:H^0(\bb{P}(L(\hat{\rho}_Y)), \mathcal{L}^*)\to H^0(\bar{X}_P, \mathcal{L}^P(\hat{\rho}_Y))$ be induced from $i:\bar{X}_P\to \bb{P}(L(\hat{\rho}_Y))$ taking $gP \mapsto [g e_{\hat{\rho}_Y}]$, where 
$e_{\hat{\rho}_Y}$ is a highest weight vector in  $L(\hat{\rho}_Y)$. 
 Set $\chi:=i^*\circ \beta: L(\hat{\rho}_Y)^\vee\xrightarrow{\sim} H^0(\bar{X}_P, \mathcal{L}^P(\hat{\rho}_Y))$, which is  the Borel-Weil isomorphism (see \cite[\S 8.1.21]{K})  given by  $\chi(f)(gP) = [g, f(ge_{\hat{\rho}_Y})]$.

	Let $e_{u\hat{\rho}_Y}$ be the extremal weight vector of $L(\hat{\rho}_Y)$ with weight $u\hat{\rho}_Y$ and $e_{u\hat{\rho}_Y}^*$ the linear form which takes value 1 on $e_{u\hat{\rho}_Y}$ and 0 on any weight vector of weight different from $u\hat{\rho}_Y$. It is easy to see (from \cite[Lemma 8.3.3 and Proposition 1.4.2(a)]{K}) that the section $\chi(e_{u\hat{\rho}_Y}^*)\vert_{X^u_P}$ has zero set exactly $\partial X^u_P$. This proves the lemma. 
	\end{proof}

Recall that a $\bb{Q}$-Cartier $\bb{Q}$-divisor $D$ on an irreducible projective variety $X$ is said to be \emph{nef} (resp. \emph{big}) if $D$ has nonnegative intersection with every irreducible curve in $X$ (resp. $\dim(H^0(X, \mathcal{O}_X(mD)))>cm^{\dim(X)}$ for some $c>0$ and $m\gg 1$). Note that if $D$ is ample, then it is \emph{nef} and \emph{big} (cf. \cite[Proposition 2.61]{KM}). 

For a proper morphism $\pi:X\to Y$ between schemes with $X$ irreducible and a $\bb{Q}$-Cartier $\bb{Q}$-divisor $D$ on $X$, $D$ is said to be $\pi$-\emph{nef} (resp. $\pi$-\emph{big}) if $D$ has nonnegative intersection with every irreducible curve in $X$ contracted by $\pi$ (resp. rank $\pi_*\mathcal{O}_X(mD)>cm^n$ for some $c>0$ and $m\gg 1$, where $n$ is the dimension of a general fiber of $\pi$). 

\begin{proposition}\label{prop:nefbig}  There exists a nef and big line bundle $\mathcal{M}$ on $(Z^{u,v}_w)_\bj$ with a section with support $\partial((Z^{u,v}_w)_\bj)$, where $\partial((Z^{u,v}_w)_\bj)$ is defined to be the inverse image of $\partial((X^{u,v}_w(P))_\bj)$ under the canonical map $(Z^{u,v}_w)_\bj\to (X^{u,v}_w(P))_\bj$ induced by the $T$-equivariant map $\pi^{u,v}_w: Z^{u,v}_w\to X^{u,v}_w(P)$. Moreover, $\mathcal{M}$ can be chosen to be the pull-back of an ample line bundle $\mathcal{M}'$ on $(X^{u,v}_w(P))_\bj$. 
	\end{proposition}

\begin{proof} Let $\mathcal{H}$ be an ample line bundle on $\bb{P}_\bj$ with a section with support precisely $\partial\bb{P}_\bj$. Let $\mathcal{L}^P_{Z^{u,v}_w}(\hat{\rho}_Y\boxtimes\hat{\rho}_Y)$ be the pull-back of the line bundle $\mathcal{L}^P(\hat{\rho}_Y)\boxtimes\mathcal{L}^P(\hat{\rho}_Y)$ on $\bar{X}_P\times\bar{X}_P$ via the morphism $Z^{u,v}_w\to \bar{X}_P\times\bar{X}_P$. The line bundle $e^{u\hat{\rho}_Y+v\hat{\rho}_Y}\mathcal{L}^P_{Z^{u,v}_w}(\hat{\rho}_Y\boxtimes\hat{\rho}_Y)$ is $T$-equivariant, hence we have the line bundle \[\tilde{\mathcal{L}}^P_{Z^{u,v}_w}(-\hat{\rho}_Y\boxtimes-\hat{\rho}_Y):=E(T)_\bj \times^Te^{u\hat{\rho}_Y+v\hat{\rho}_Y}\mathcal{L}^P_{Z^{u,v}_w}(\hat{\rho}_Y\boxtimes\hat{\rho}_Y)\to (Z^{u,v}_w)_\bj.\]
	
Set $\mathcal{M}:=\tilde{\mathcal{L}}^P_{Z^{u,v}_w}(-\hat{\rho}_Y\boxtimes-\hat{\rho}_Y)\otimes\pi^*(\mathcal{H}^{N'})$, where $\pi$ is the standard quotient map $E(T)_\bj \times^T Z^{u,v}_w\to \bb{P}_\bj$ and $N'\gg0$. Take $\theta:(Z^{u,v}_w)_\bj\to \tilde{\mathcal{L}}^P_{Z^{u,v}_w}(-\hat{\rho}_Y\boxtimes-\hat{\rho}_Y)$ to be the section given by \[[e,z]\mapsto [e,1_{u\hat{\rho}_Y+v\hat{\rho}_Y}\otimes (\bar{\chi}(e^*_{u\hat{\rho}_Y})\boxtimes\bar{\chi}(e^*_{v\hat{\rho}_Y}))(z)],\] where $e\in E(T)_\bj,\ z\in Z^{u,v}_w,\ 1_{u\hat{\rho}_Y+v\hat{\rho}_Y}$ is the constant section of the trivial line bundle over $Z^{u,v}_w$ with $H$ action given by $u\hat{\rho}_Y+v\hat{\rho}_Y$, and $\bar{\chi}\boxtimes\bar{\chi}$ is the pull-back of the Borel-Weil isomorphism as in the proof of \Cref{sec:support}. Let $\sigma$ be any section of $\mathcal{H}^{N'}$ with support precisely $\partial\bb{P}_\bj$ and $\hat{\sigma}$ its pull-back to $(Z^{u,v}_w)_\bj$. It is easy to see that the support of $\theta\otimes \hat{\sigma}$ is precisely $\partial((Z^{u,v}_w)_\bj)$. 

Moreover, the line bundle $\mathcal{M}$ is the pull-back of the line bundle $\mathcal{M}':=\tilde{\mathcal{L}}'(-\hat{\rho}_Y\boxtimes-\hat{\rho}_Y)\otimes\pi_1^*(\mathcal{H}^{N'})$ on $E(T)_\bj\times^T X^{u,v}_w(P)$ via the morphism 
\[\pi: E(T)_\bj\times^T Z^{u,v}_w\to E(T)_\bj\times^T X^{u,v}_w(P),\] where $\pi_1$ is the standard quotient map $E(T)_\bj\times^TX^{u,v}_w(P)\to\bb{P}_\bj$ and \[\tilde{\mathcal{L}}'(-\hat{\rho}_Y\boxtimes-\hat{\rho}_Y):=E(T)_\bj\times^T \left(e^{u\hat{\rho}_Y+v\hat{\rho}_Y}(\mathcal{L}^P(\hat{\rho}_Y)\boxtimes\mathcal{L}^P(\hat{\rho}_Y))\vert_{X^{u,v}_w(P)}\right).\]
From \cite[Proposition 1.45 and Theorems 1.37 and 1.42]{KM}, $\mathcal{M}'$ is ample on $(X^{u,v}_w(P))_\bj$ for large enough $N'>0$. Since $\pi$ is a birational morphism and $\mathcal{M}'$ is ample, $\mathcal{M}$ is nef and big by \cite[\S 1.29]{D}.
	\end{proof}

\begin{proposition}\label{thm:kv} For $\tilde{\pi}:\widetilde{\mathcal{Z}}\to \bar{\Gamma}$, \[R^p\tilde{\pi}_*(\omega_{\widetilde{\mathcal{Z}}}(\partial\widetilde{\mathcal{Z}}))= 0\hspace{5mm}\text{for all} \,\, p>0,\] where $\partial\widetilde{\mathcal{Z}}:=f^{-1}(\partial \mathcal{Z}_P)$ and $\omega_{\widetilde{\mathcal{Z}}}(\partial\widetilde{\mathcal{Z}}):=\shom_{\ss{\widetilde{\mathcal{Z}}}}(\idshf{\widetilde{\mathcal{Z}}}{\partial\widetilde{\mathcal{Z}}},\omega_{\widetilde{\mathcal{Z}}}).$ Here we take $\partial \mathcal{Z}_P$ with the reduced scheme structure.

(Observe that $f$ being a desingularization of a normal scheme $\mathcal{Z}_P$ and $\partial \mathcal{Z}_P$ being reduced, $\partial \widetilde{\mathcal{Z}}$ is also a reduced scheme.)	
\end{proposition}
\begin{proof} As guaranteed by  \Cref{prop:nefbig}, let $\mathcal{M}$ be a nef and big line bundle on $(Z^{u,v}_w)_\bj$ with divisor $\sum_{i=1}^d b_i Z_i$ (where $b_i>0$ for all $i$) supported precisely in $\partial((Z^{u,v}_w)_\bj)$. Further,  $\mathcal{M}$ can be taken to be  the pull-back of an ample line bundle $\mathcal{M}'$ on $(X^{u,v}_w(P))_\bj$. Let $N_1$ be an integer so that $N_1>b_i$ for all $i$. 
	
	By the proof of Proposition \ref{prop:nefbig}, the line bundle $\epsilon\boxtimes \mathcal{M}'$ over $\bar{\Gamma}\times (X^{u,v}_w(P))_\bj$ restricted to $\mathcal{Z}_P$ via $i$ has a section with support $\partial \mathcal{Z}_P$, where $\partial \mathcal{Z}_P$ is defined in Proposition \ref{prop:bdyZ} and $\epsilon$ is the trivial line bundle over $\bar{\Gamma}$. Hence, the pull-back of the line bundle 
	$f^*(i^*(\epsilon \boxtimes \mathcal{M}'))$ over 	$\widetilde{\mathcal{Z}}$ has a section with support $f^{-1}(\partial \mathcal{Z}_P)=\partial \widetilde{\mathcal{Z}}$.			
	Since $\partial\widetilde{\mathcal{Z}}$ is the zero set of a line bundle on $\widetilde{\mathcal{Z}}$, $\partial\widetilde{\mathcal{Z}}$ is a pure scheme of codimension 1 in $\widetilde{\mathcal{Z}}$. Let $\mathcal{L}$ be the line bundle on the smooth scheme $\widetilde{\mathcal{Z}}$ associated to the reduced divisor $\partial\widetilde{\mathcal{Z}}$ and let $D$ be the divisor $\sum_i (N_1-b_i)\widetilde{Z}_i$ on $\widetilde{\mathcal{Z}}$, where \[\widetilde{\mathcal{Z}}_i:=(\bar{\Gamma}\times Z_i)\times_{(Z^2_w)_\bb{P}} \tilde{\Delta}((Z_w)_\bb{P}).\] 
	Then, as proved in \cite[Proof of Proposition 8.4]{Ku},  each $\widetilde{Z}_i$ is a smooth irreducible divisor of $\widetilde{\mathcal{Z}}$, and for any collection $\widetilde{Z}_{i_1},\cdots, \widetilde{Z}_{i_q}$, $1\leq i_1<\cdots <i_q\leq d$, the intersection $\bigcap_{j = 1}^q \widetilde{Z}_{i_j}$ (if non-empty) is smooth and of pure codimension $q$ in $\widetilde{\mathcal{Z}}$. Observe that $\widetilde{\mathcal{Z}}$ here coincides with the same in \cite[Proposition 8.4]{Ku}. Moreover, the collection $\{\widetilde{\mathcal{Z}}_i\}$ of divisors here is a subcollection of the divisors in 	\cite[Proposition 8.4]{Ku} since the inverse image of 	 $\partial ((X^{u,v}_w(P))_\bj)$	in $(X_w^{u, v})_\bj$ is union of certain irreducible divisors in 	 $\partial ((X^{u,v}_w)_\bj)$.	
	Hence, the $\widetilde{Z}_i$'s are all distinct and $\partial\widetilde{\mathcal{Z}} = \sum \widetilde{Z}_i$ is a simple normal crossing divisor. Clearly, \[\mathcal{L}^{N_1}(-D) =\ss{\widetilde{\mathcal{Z}}}\left(\sum_i b_i\widetilde{Z}_i\right)\simeq \tilde{i}^*\left(\ss{\bar{\Gamma}\times (Z^{u,v}_w)_\bj}\left(\sum_i b_i(\bar{\Gamma}\times Z_i)\right)\right).\]
	By \cite[\S 1.6]{D}, since $\sum_i b_iZ_i$ is a nef divisor (by assumption), and since $\tilde{i}$ is injective, $\mathcal{L}^{N_1}(-D)$ is $\tilde{\pi}$-nef. 
	
	We next show that $\mathcal{L}^{N_1}(-D)$ is $\tilde{\pi}$-big. 
	Since $\mathcal{M}$ was chosen to be the pull-back of an ample line bundle $\mathcal{M}'$ on $(X^{u,v}_w(P))_\bj$, we have that $\mathcal{L}^{N_1}(-D)$ is the pull-back of the line bundle $\mathcal{S}:=i^*(\epsilon\boxtimes \mathcal{M}')$ on $\mathcal{Z}_P$ via $f$. But $\mathcal{M}'$ being ample implies $\mathcal{S}$ is $\pi$-big. Further, since $f$ is birational, the fibers of $\tilde{\pi}$ for general $\gamma$ have the same dimension as the fibers of $\pi$ (use \cite[Ch. I, $\S$6.3, Theorem 1.25]{S}). Hence $\mathcal{L}^{N_1}(-D)$ is  $\tilde{\pi}$-big. 
	
	Since $f$ is proper and birational it is surjective. We also have that $\pi$ is surjective (cf. Corollary \ref{coro8.5}), hence so is $\tilde{\pi}$. Now apply \cite[Theorem 8.3]{Ku}.
	\end{proof}

\begin{theorem}\label{acyclic} For the morphism $f:\widetilde{\mathcal{Z}}\to \mathcal{Z}_P$,
	\begin{enumerate}\item[(a)] $R^pf_*(\omega_{\widetilde{\mathcal{Z}}}(\partial \widetilde{\mathcal{Z}}))=0$ for all $p>0$, and
		\item[(b)] $f_*(\omega_{\widetilde{\mathcal{Z}}}(\partial\widetilde{\mathcal{Z}})) = \omega_{\mathcal{Z}_P}(\partial \mathcal{Z}_P)$,
		where  $ \omega_{\mathcal{Z}_P}(\partial \mathcal{Z}_P):= 	\shom_{\ss{{\mathcal{Z}_P}}}(\idshf{\mathcal{Z}_P}{\partial\mathcal{Z}_P},\omega_{\mathcal{Z}_P}).$
\end{enumerate}
	\end{theorem}
\begin{proof} For (a), note that $f$ is surjective as observed above. As in the proof of the previous proposition, $\mathcal{L}^N(-D)$ is $\tilde{\pi}$-nef and $\tilde{\pi}$-big. Now the fibers of $\tilde{\pi}$ contain the fibers of $f$, so that $\mathcal{L}^N(-D)$ is also $f$-nef. Since $f$ is birational, we also have that $\mathcal{L}^N(-D)$ is $f$-big. Then (a) follows from \cite[Theorem 8.3]{Ku}. 
	
	(b) Since $\partial\widetilde{\mathcal{Z}}=f^{-1}(\partial\mathcal{Z}_P)$ is the scheme-theoretic inverse image with the reduced scheme structure on $\partial \mathcal{Z}_P$, the morphism \[f^*(\idshf{\mathcal{Z}_P}{\partial\mathcal{Z}_P})\to \idshf{\widetilde{\mathcal{Z}}}{\partial\widetilde{\mathcal{Z}}}\] is surjective 
	(cf. \cite[Tag 01HJ, Lemma 25.4.7]{stack}) with kernel supported on a proper closed subset of $\widetilde{\mathcal{Z}}$ (since $f$ is a desingularization). Thus the kernel is a torsion sheaf and so the dual map \begin{equation}\label{dualhom}\ss{\widetilde{\mathcal{Z}}}(\partial\widetilde{\mathcal{Z}})\to \shom_{\ss{\widetilde{\mathcal{Z}}}}(f^*(\idshf{\mathcal{Z}_P}{\partial\mathcal{Z}_P}),\ss{\widetilde{\mathcal{Z}}})
	\end{equation} is an isomorphism, where $\ss{\widetilde{\mathcal{Z}}}(\partial\widetilde{\mathcal{Z}}):=\shom_{\ss{\widetilde{\mathcal{Z}}}}(\idshf{\widetilde{\mathcal{Z}}}{\partial\widetilde{\mathcal{Z}}},\ss{\widetilde{\mathcal{Z}}})$. 
	
	Finally, \begin{align*} f_*(\omega_{\widetilde{\mathcal{Z}}}(\partial\widetilde{\mathcal{Z}})) & = f_*(\omega_{\widetilde{\mathcal{Z}}}\otimes \shom_{\ss{\widetilde{\mathcal{Z}}}}(\idshf{\widetilde{\mathcal{Z}}}{\partial\widetilde{\mathcal{Z}}},\ss{\widetilde{\mathcal{Z}}}))\\
		& = f_*(\omega_{\widetilde{\mathcal{Z}}}\otimes \shom_{\ss{\widetilde{\mathcal{Z}}}}(f^*(\idshf{\mathcal{Z}_P}{\partial\mathcal{Z}_P}),\ss{\widetilde{\mathcal{Z}}}))\hspace{5mm} \text{by \eqref{dualhom}}\\
		& = f_*\shom_{\ss{\widetilde{\mathcal{Z}}}}(f^*(\idshf{\mathcal{Z}_P}{\partial\mathcal{Z}_P}),\omega_{\widetilde{\mathcal{Z}}})\\
		& =\shom_{\ss{{\mathcal{Z}_P}}}(\idshf{\mathcal{Z}_P}{\partial\mathcal{Z}_P},f_*\omega_{\widetilde{\mathcal{Z}}})\hspace{5mm} \text{by adjunction \cite[Ch. II, \S 5]{H}}\\
		& = \shom_{\ss{{\mathcal{Z}_P}}}(\idshf{\mathcal{Z}_P}{\partial\mathcal{Z}_P},\omega_{\mathcal{Z}_P}) \hspace{5mm} \text{by \Cref{ratsings} and \cite[Theorem 5.10]{KM} }\\
		& = \omega_{\mathcal{Z}_P}(\partial\mathcal{Z}_P).\qedhere
		\end{align*}
	\end{proof}

The following is an immediate consequence of Proposition \ref{thm:kv}, 
\Cref{acyclic} and the Grothendieck spectral sequence \cite[Part I, Proposition 4.1]{J} applied to $\tilde{\pi}=\pi\circ f$.

\begin{corollary} \label{coro9.5} The morphism $\pi:\mathcal{Z}_P\to \bar{\Gamma}$ from the diagram satisfies \[R^p\pi_*(\omega_{\mathcal{Z}_P}(\partial\mathcal{Z}_P))=0\hspace{5mm} \text{for all}\ p>0.\]
	\end{corollary}

\section{Proof of Part (B) of \Cref{main}}

From \Cref{prop:finunion}, a similar argument to the proof of part (A) of \Cref{main} gives the vanishing \[\stor_1^{\ss{(\bar{Y}_P)_\bb{P}}}(\gamma_*\tilde{\Delta}_*\ss{(X_w^P)_\bb{P}}, \ss{\partial((X^{u,v}_P)_\bj)}) = 0\hspace{5mm} \text{for general}\ \gamma\in \bar{\Gamma}.\] 
This together with the definition that $\xi^u_P := \idshf{X^u_P}{\partial X^u_P}$ implies that part (B) of \Cref{main} is equivalent to the following
\begin{theorem} Assume $c_{u,v}^w(\bj)\neq 0$. For general $\gamma\in\bar{\Gamma}$, 
	\[H^p(\bar{N}_\gamma, \mathcal{O}_{\bar{N}_\gamma}(-\bar{M}_\gamma)) = 0\hspace{5mm} \text{for all}\ p\neq |\bj|+l(w)-l(u)-l(v),\]
	where $\bar{M}_\gamma:=M_{\gamma^{-1}}$ is the subscheme $(\partial((X^{u,v}_P)_\bj))\cap \gamma\tilde{\Delta}((X_w^P)_\bb{P})$ and $\mathcal{O}_{\bar{N}_\gamma}(-\bar{M}_\gamma)$ is the ideal sheaf of $\bar{M}_\gamma$ in $\bar{N}_\gamma:=(X^{u,v}_P)_\bj\cap\gamma\tilde{\Delta}((X_w^P)_\bb{P})$. 
	\end{theorem}
\begin{proof} Since $\mathcal{Z}_P$ and $\partial\mathcal{Z}_P$ are Cohen-Macaulay and $\partial\mathcal{Z}_P$ is of codimension 1 in $\mathcal{Z}_P$ (cf. Propositions \ref{prop8.1} and \ref{prop:bdyZ}), we have \[\sext^i_{\ss{\mathcal{Z}_P}}(\idshf{\mathcal{Z}_P}{\partial\mathcal{Z}_P},\omega_{\mathcal{Z}_P}) = 0\hspace{5mm} \text{for all}\ i\geq 1.\] Similarly, by the identity \eqref{eqn21},
we also have for general $\gamma\in\bar{\Gamma}$,\[\sext^i_{\ss{\bar{N}_\gamma}}(\idshf{\bar{N}_\gamma}{\bar{M}_\gamma},\omega_{\bar{N}_\gamma}) = 0\hspace{5mm}\text{for all}\ i\geq 1.\] 
	By the Serre duality \cite[Ch. III, Theorem 7.6]{H} applied to $\bar{N}_\gamma$ and the local-to-global Ext spectral sequence \cite[Ch. II, Th\'{e}or\`{e}me 7.3.3]{Go}, the theorem is equivalent to the vanishing 
	\begin{equation}\label{van:coh} H^p(\bar{N}_\gamma,\shom_{\ss{\bar{N}_\gamma}}(\idshf{\bar{N}_\gamma}{\bar{M}_\gamma}, \omega_{\bar{N}_\gamma})) = 0\hspace{5mm}\text{for all}\ p>0,
		\end{equation}
	since $\bar{N}_\gamma$ is Cohen-Macaulay (for general $\gamma\in\bar{\Gamma}$) and $\dim(\bar{N}_\gamma)=|\bj|+l(w)-l(u)-l(v)$ (cf. Corollary 
\ref{coro8.5}).	
	
	 Let $\omega_{\bar{N}_\gamma}(\bar{M}_\gamma):=\shom_{\ss{\bar{N}_\gamma}}(\idshf{\bar{N}_\gamma}{\bar{M}_\gamma},\omega_{\bar{N}_\gamma})$. Then for general $\gamma\in\bar{\Gamma}$, \begin{equation}\label{canshf} \omega_{\mathcal{Z}_P}(\partial \mathcal{Z}_P)\vert_{\pi^{-1}(\gamma^{-1})}\simeq \omega_{\pi^{-1}(\gamma^{-1})}(\partial \mathcal{Z}_P\cap \pi^{-1}(\gamma^{-1})) = \omega_{\bar{N}_\gamma}(\bar{M}_\gamma).
	 	\end{equation}
 	To prove this, observe that by \cite[Ch. I, \S 6.3, Theorem 1.25]{S} and \cite[Ch. III, Exercise 10.9]{H}, there exists an open nonempty subset $\bar{\Gamma}_o\subset \bar{\Gamma}$ such that $\pi:\pi^{-1}(\bar{\Gamma}_o)\to \bar{\Gamma}_o$ is a flat morphism. Since $\bar{\Gamma}_o$ is smooth and $\mathcal{Z}_P$ and $\partial\mathcal{Z}_P$ are Cohen-Macaulay, and since the assertion is local in $\bar{\Gamma}$, it is enough to observe (see \cite[Corollary 11.35]{I}) that for a nonzero function $\theta$ on $\bar{\Gamma}_o$, there is an isomorphism of $\ss{\mathcal{Z}_P^\theta}$-modules
 	\[\mathcal{S}/\theta\cdot \mathcal{S}\simeq \shom_{\ss{\mathcal{Z}_P^\theta}}(\idshf{\mathcal{Z}_P}{\partial\mathcal{Z}_P}/\theta\cdot \idshf{\mathcal{Z}_P}{\partial\mathcal{Z}_P},\omega_{\mathcal{Z}_P^\theta}),\]
 	where $\mathcal{Z}_P^\theta$ is the zero scheme of $\theta$ in $\mathcal{Z}_P$ and $\mathcal{S}:=\shom_{\ss{\mathcal{Z}_P}}(\idshf{\mathcal{Z}_P}{\partial\mathcal{Z}_P},\omega_{\mathcal{Z}_P})$. Taking $\theta$ in a local coordinate system, we can continue and get \eqref{canshf}.
 	
 	Now $R^p\pi_*(\omega_{\mathcal{Z}_P}(\partial\mathcal{Z}_P)) =0$ for $p>0$ (cf. Corollary \ref{coro9.5}) 
implies that for general $\gamma\in\bar{\Gamma}$, \eqref{van:coh} holds. This follows from the fact that since $\mathcal{Z}_P$ and $\partial\mathcal{Z}_P$ are Cohen-Macaulay, $\bar{\Gamma}_o$ is smooth, and $\pi:\pi^{-1}(\bar{\Gamma}_o)\to\bar{\Gamma}_o$ is flat, we have $\omega_{\mathcal{Z}_P}(\partial \mathcal{Z}_P)$ is flat over $\bar{\Gamma}_o$ (see the proof of \cite[Theorem 9.1]{Ku}). Therefore \eqref{van:coh} follows from the semicontinuity theorem (\cite[Ch. III, Theorem 12.8 and Corollary 12.9]{H}). Hence the theorem is proven and therefore so is part (B) of \Cref{main}. 	Thus, Corollary \ref{coeff:calts}	is proved. 
\end{proof}

\newpage

\section{Appendix}

The aim of this appendix is to determine the dualizing sheaf $\omega_{X^w_P}$ of the Cohen-Macaulay scheme $X^w_P \subset \bar{X}_P$  for any $w\in W^P$ (in the thick Kac-Moody flag variety). Even though this result is not used in the paper, we believe that it is interesting on its own for potentially its future use.  When $P=B$, this result was proved by Kashiwara (cf. \cite[Theorem 10.4]{Ku}). Moreover, in the finite case (i.e., when $G$ is a semisimple group), this result for any $P$ is obtained in \cite[Theorem 3.3]{KRW}.

From \cite[Part II, \S 4.2]{J}, the dualizing sheaf of $\bar{X}_P$ in the finite case is given by \[\omega_{\bar{X}_P}=\mathcal{L}^P(-2\rho+2\rho^Y),\]  where $\rho^Y$ denotes the half sum of positive roots coming from the Levi component of $P$. 

Given any Cohen-Macaulay subscheme $Y$ of a smooth variety $X$, its dualizing sheaf $\omega_Y$ is given by
\[\omega_Y\simeq\sext_{\ss{X}}^{\text{codim}\ Y}(\ss{Y},\ss{X})\otimes\omega_X.\] 
In particular, by \Cref{prop:norm}, following the analogy with the schemes of finite type, we define the dualizing sheaf of $X^w_P$ (for any $w\in W^P$) by
\[\omega_{X^w_P} := \sext_{\ss{\bar{X}_P}}^{l(w)}(\ss{X^w_P},\ss{\bar{X}_P})\otimes\mathcal{L}^P(-2\rho+2\rho^Y).\]

\noindent Recall the definitions of $\rho_Y = \sum_{i\in Y} \varpi_i$ and $\hat{\rho}_Y = \rho-\rho_Y$ from Section 2. For any $w\in W^P$, let \begin{align*}\hat{\xi}^w_P:&=\bb{C}_{-\rho+w\rho_Y}\otimes\omega_{X^w_P}\otimes\mathcal{L}^P(\rho+\rho_Y-2\rho^Y)\\ &= \bb{C}_{-\rho+w\rho_Y}\otimes \sext_{\ss{\bar{X}_P}}^{l(w)}(\ss{X^w_P},\ss{\bar{X}_P})\otimes \mathcal{L}^P(-\hat{\rho}_Y).
\end{align*}

\begin{lemma}\label{lem:rest} Restricted to the open cell $C^w_P\subset X^w_P$, we have a $B^-$-equivariant isomorphism 
	\[(\hat{\xi}^w_P)\vert_{C^w_P}\simeq (\ss{X^w_P})\vert_{C^w_P}.\]
	\end{lemma}
\begin{proof}Because of the $B^-$-equivariance of $\hat{\xi}^w_P$, it suffices to show that $i^*_w\hx$ is trivial as an $H$-module, where $i_w:\{pt\}\to \bar{X}_P$ is the map sending the point to the fixed point $w$. First, note that for any character $\lambda$ of $P$, $i^*_w\mathcal{L}^P(\lambda) = \bb{C}_{-w\lambda}$. Further, \[i^*_w\left(\sext_{\ss{\bar{X}_P}}^{l(w)}\left(\ss{X^w_P},\ss{\bar{X}_P}\right)\right) \simeq \det\left(\frac{T_w(\bar{X}_P)}{T_w(X^w_P)}\right) \simeq \bb{C}_{\rho-w\rho}.\]
	The last equality follows from the computation of the tangent spaces: 
	\begin{align*}T_w(X^w_P) &= T_w(B^-wP/P) = T_w\left(w(w^{-1}B^-w\cap U^-_P)P/P\right) = \bigoplus_{\beta\in \Delta^-\cap w(\Delta^-\setminus \Delta^-_Y)}\mf{g}_\beta\\
	T_w(\bar{X}_P) &= \bigoplus_{\beta\in w(\Delta^-\setminus \Delta_Y^-)}\mf{g}_\beta,
	\end{align*}
		where $\Delta^-$ is the set of negative roots, $\Delta^-_Y$ the set of negative roots of the Levi component of $P$, and $U^-_P$ is the unipotent radical of the opposite parabolic of $P$. Then by \cite[Corollary 1.3.22]{K},
	
	 \[\det\left(\frac{T_w(\bar{X}_P)}{T_w(X^w_P)}\right) = \det\left(\bigoplus_{\beta\in \Delta^+\cap w(\Delta^-\setminus \Delta^-_Y)}\mf{g}_\beta\right)=\det\left(\bigoplus_{\beta\in \Delta^+\cap w\Delta^-}\mf{g}_\beta\right)= \bb{C}_{\rho-w\rho},\,\,\,\text{since $w\in W^P$}.\] Therefore the conclusion of the lemma follows since the weight of $i^*_w\hx$ is given by \[-\rho+w\rho_Y + (\rho-w\rho) - w(-\hat{\rho}_Y) = w\rho_Y-w\rho+w\rho-w\rho_Y = 0.\qedhere\]
	\end{proof}

Let $\ds V^w_P :=C^w_P\cup \bigcup_{v\leftarrow w} C^v_P$, where the notation $v\leftarrow w$ indicates that $l(w) = l(v) -1$ and $v> w \in W^P$. Then, $V^w_P$ is a smooth open subset of $X^w_P$ and $\hx\vert_{V^w_P}$ is an invertible $B^-$-equivariant $\ss{V^w_P}$-module. By \Cref{lem:rest}, \[\hx\vert_{V^w_P}\simeq \ss{X^w_P}\left(-\sum_{v\leftarrow w} m_{w,v}^PX^v_P\right)\vert_{V^w_P}, \hspace{5mm}\text{for some}\ m_{w,v}^P\in\bb{Z}. \]

\begin{lemma}\label{lem:m} The coefficients $m_{w,v}^P$ are given by the formula \[m_{w,v}^P =1- \id{w\rho_Y, \beta^\vee}, \] where $\beta$ is the positive root so that $s_\beta w = v$. 
	\end{lemma}
\begin{proof}Let $v\in W^P$ be such that $v\leftarrow w$. We first compute \begin{align*}i^*_v\left(\sext_{\ss{\bar{X}_P}}^{l(w)}(\ss{X^w_P},\ss{\bar{X}_P})\right) &\simeq \det\left(\frac{T_v(\bar{X}_P)}{T_v(X^w_P)}\right) \\ & \simeq \det\left(\frac{T_v(\bar{X}_P)}{T_v(X^v_P)}\right)\otimes \det\left(\frac{T_v(X^w_P)}{T_v(X^v_P)}\right)^*\\ &\simeq \bb{C}_{\rho-v\rho-\beta},\,\,\,\text{by \cite[Lemma 10.3]{Ku}}.
		\end{align*}
	Hence, \begin{align}\label{onecalc}i^*_v\hx & \simeq \bb{C}_{-\rho+w\rho_Y}\otimes \bb{C}_{\rho-v\rho-\beta}\otimes \bb{C}_{v\hat{\rho}_Y}\nonumber \\ & = \bb{C}_{w\rho_Y-\beta-v\rho_Y}\nonumber \\ & =\bb{C}_{(\id{w\rho_Y,\beta^\vee}-1)\beta}.
		\end{align}
	On the other hand \begin{align}\label{another}i^*_v\left(\ss{X^w_P}\left(-\sum_{u\leftarrow w} m_{w,u}^PX^u_P\right)\right) & = \det\left(\frac{T_v(X^w_P)}{T_v(X^v_P)}\right)^{\otimes -m_{w,v}^P}\nonumber \\ & = \bb{C}_\beta^{\otimes -m_{w,v}^P}\nonumber \\ & = \bb{C}_{-m_{w,v}^P\beta}.
		\end{align}
	Equating the equations  \eqref{onecalc} and \eqref{another}, we have the desired result. 
	\end{proof}

\begin{theorem}For any $w\in W^P$, we have a $B^-$-equivariant isomorphism \[\hx\simeq \ss{X^w_P}\left(-\sum_{v\leftarrow w \in W^P} m_{w,v}^PX^v_P\right),\] where $m_{w,v}^P$ are as in \Cref{lem:m}. Therefore, the dualizing sheaf of $X^w_P$ is $T$-equivariantly isomorphic to
	\[\bb{C}_{\rho-w\rho_Y}\otimes \ss{X^w_P}\left(-\sum_{v\leftarrow w} m_{w,v}^PX^v_P\right)\otimes \mathcal{L}^P(2\rho^Y-\rho-\rho_Y).\]
	\end{theorem}
\begin{proof}Let $j:V^w_P\hookrightarrow X^w_P$ be the inclusion map and set $D:=\ds \sum_{v\leftarrow w} m_{w,v}^P X^v_P$. We have the following commutative diagram with exact rows \begin{center}
	\begin{tikzcd}0\arrow[r] & \idshf{X^w_P}{D}\arrow[r]\arrow[d] & \ss{X^w_P}\arrow[r]\arrow[d,"\simeq"] &\ss{D}\arrow[r]\arrow[d,hook] & 0\\
		0\arrow[r] & j_*j^{-1}(\idshf{X^w_P}{D})\arrow[r] & j_*j^{-1}(\ss{X^w_P})\arrow[r] & j_*j^{-1}\ss{D}. & 
		\end{tikzcd}\end{center}
	Here the middle vertical arrow is an isomorphism since $X^w_P$ is normal (cf. Lemma \ref{prop:norm})	
	and $X^w_P\setminus V^w_P$ is of codimension at least 2 in $X^w_P$. The right vertical map is injective since $\overline{\text{supp}(D)\cap V^w_P} = \text{supp}(D)$. This implies that the leftmost vertical map is an isomorphism. We also have that $\hx\simeq j_*j^{-1}(\hx)$ since $\hx$ is a Cohen-Macaulay $\ss{X^w_P}$-module. Moreover, by Lemmas \ref{prop:norm}	
	and \ref{lem:m},  \[\hx\simeq j_*j^{-1}(\hx)\simeq j_*j^{-1}\idshf{X^w_P}{D}\simeq\idshf{X^w_P}{D}.\qedhere\]
	\end{proof}

For $P=B$, clearly $\rho_Y=0$ and the coefficients $m_{w,v}^P = 1- \id{w\rho_Y, \beta^\vee}=1$. This gives the following theorem of Kashiwara (\cite[Theorem 10.4]{Ku})
\begin{theorem}For any $w\in W$, we have a $B^-$-equivariant isomorphism  \[\xi^w\simeq \idshf{X^w}{\partial X^w},\] where $\xi^w:=\bb{C}_{-\rho}\otimes\mathcal{L}(-\rho)\otimes \sext_{\ss{\bar{X}}}^{l(w)}(\ss{X^w},\ss{\bar{X}}).$
	\end{theorem}

\begin{example} (a) {\rm For the affine Kac-Moody group $G= \widehat{G_o}$, where $G_o$ is a simple simply-connected complex algebraic group, taking the standard maximal parabolic subgroup $P$, we see that if we take $w=e, v=s_0$, then $m^P_{w, v}=1.$}
	
	\vskip1ex
	(b) {\rm Let us take  $G=\widehat{SL_2}$ and $P$ the standard maximal parabolic subgroup (as above). Then, for $w=s_0, v=s_1s_0$,  $m^P_{w, v}=0.$ If we take $w=s_1s_0, v=s_0s_1s_0$, then  $m^P_{w, v}=-1.$}

\end{example}

\newpage

\bibliographystyle{alpha}
\bibliography{kpositivityIII}

\vskip3ex
\noindent
J. C. and S.K.: Department of Mathematics, University of North Carolina,
Chapel Hill, NC 27599-3250, USA (emails: compja$@$live.unc.edu, shrawan$@$email.unc.edu)

\end{document}